\documentclass[final,leqno]{siamltex}   %SIAM

\usepackage{color}
\usepackage{array,graphicx}
\usepackage{psfrag}
\usepackage{subfigure}
\usepackage{url}
\usepackage{cite}
\usepackage{epsfig}
\usepackage{amsmath,amssymb}

\providecommand{\rset}[1]{\mathbb{R}^}
\providecommand{\abs}[1]{\lvert#1\rvert}
\providecommand{\norm}[1]{\lVert#1\rVert}

\newtheorem{assumption}[theorem]{Assumption}

\newtheorem{example}[theorem]{Example}

\newenvironment{algorithm}[1]
{\vskip0.1cm\noindent\textsc{Algorithm } $($#1$)$.}{}
\usecounter{algorithmenumi}

\title{ Iteration complexity analysis of random coordinate descent methods
for $\ell_0$ regularized convex problems }
\author{  Andrei Patrascu and  Ion Necoara
\thanks{I. Necoara and  A. Patrascu are with  University  Politehnica Bucharest,
Automatic Control and Systems Engineering Department, 060042
Bucharest, Romania. {\tt\small \{ion.necoara,andrei.patrascu\}@acse.pub.ro}.}}
\date{\normalsize April 2014}

\begin{document}
\begin{sloppy}

%JOTA
%\title{ Iteration complexity analysis of random coordinate descent methods
%for $\ell_0$ regularized convex problems}
%\author{Ion Necoara and Andrei Patrascu}
%\institute{I. Necoara and A. Patrascu  are with Automatic Control
%and Systems Engineering Department, University Politehnica
%Bucharest, 060042 Bucharest, Romania. Corresponding author:
%A.~Patrascu, email: andrei.patrascu@acse.pub.ro.}
%\titlerunning{Coordinate descent methods for $\ell_0$
% regularized optimization}
%\authorrunning{ I. Necoara  \and A. Patrascu }

\maketitle

\begin{abstract}
In this paper we analyze a family of general random block coordinate
descent  methods for the minimization of $\ell_0$ regularized
optimization problems, i.e. the objective function is composed of a
smooth convex function and the $\ell_0$ regularization.  Our family
of methods covers particular cases such as random block coordinate
gradient descent and random proximal coordinate descent methods. We
analyze necessary optimality conditions for this nonconvex $\ell_0$
regularized problem and devise a separation of the set of local
minima into restricted classes based on approximation versions of
the objective function. We provide a unified analysis of the almost
sure convergence for this family of block coordinate descent
algorithms and prove that, for each approximation version, the limit
points are local minima from the corresponding restricted class of
local minimizers. Under the strong convexity assumption, we prove
linear convergence in probability for our family of methods.
\end{abstract}

%SIAM
\begin{keywords}
$\ell_0$ regularized convex problems, Lipschitz gradient, restricted classes of local minima, random coordinate descent methods, iteration complexity analysis.
\end{keywords}
\pagestyle{myheadings} \thispagestyle{plain} \markboth{A. Patrascu  and I. Necoara}{Coordinate descent methods for $\ell_0$ regularized optimization}

%%%%%%%%%%%%%%%%%%%%%%%%%%%%%%%%%%%%%%%%%%%%%%%%%%%%

\section{Introduction}
\noindent In this paper we analyze the properties of local minima
and devise a family of random block coordinate descent methods for
the following $\ell_0$ regularized optimization problem:
\begin{equation}\label{l0regular}
\min\limits_{x \in \rset^n} F(x) \quad \left(= f(x) +
\norm{x}_{0,\lambda} \right),
\end{equation}
where  function $f$ is smooth and convex and the quasinorm of $x$ is
defined as:  \[ \norm{x}_{0,\lambda}= \sum\limits_{i=1}^N
\lambda_i\norm{x_i}_0, \] where $\|x_i\|_0$ is the quasinorm which
counts the number of nonzero components in the vector $x_i \in
\rset^{n_i}$, which is the $i$th block component of $x$, and
$\lambda_i \ge 0$ for all $i=1, \dots, N$. Note that in this
formulation we do not impose sparsity on all block components of
$x$, but only on those $i$th blocks for which the corresponding
penalty parameter $\lambda_i>0$.  However, in order to avoid the
convex case, intensively studied in the literature, we assume that
there is at least one $i$ such that $\lambda_i>0$.

\noindent In many applications such as  compressed sensing
\cite{BluDav:08,CanTao:04}, sparse support
vector machines \cite{Bah:13}, sparse nonnegative factorization
\cite{Gil:12}, sparse principal component analysis \cite{JouNes:10}
or robust estimation \cite{Kek:13} we deal with a convex
optimization problem for which we like to get an (approximate)
solution, but we also desire a solution which has the additional
property of sparsity (it has few nonzero components). The typical
approach for obtaining a sparse minimizer of an optimization problem
involves minimizing the number of nonzero components of the
solution. In the literature for sparse optimization  two
formulations are widespread: \textit{(i) the regularized
formulation} obtained by adding an $\ell_0$ regularization term to
the original objective function as in \eqref{l0regular};
\textit{(ii) the sparsity constrained formulation} obtained by
including an additional constraint on the number of nonzero elements
of the variable vector. However, both formulations are hard
combinatorial problems, since solving them exactly would require to
try all possible sparse patterns in a brute-force way. Moreover,
there is no clear equivalence between them in the general case.

\noindent Several greedy algorithms have been developed in the last
decade  for the sparse linear least squares setting under certain
restricted isometry assumptions \cite{Bah:13,BluDav:08,CanTao:04}.
In particular,  the iterative hard thresholding algorithm has gained
a lot of interest lately due to its simple  iteration
\cite{BluDav:08}. Recently, in \cite{Lu:12}, a generalization of the
iterative hard thresholding algorithm has been given for general
$\ell_0$ regularized convex cone programming. The author shows
linear convergence of this algorithm for strongly convex  objective
functions,  while for general convex objective functions the author
considers the minimization over a bounded box set. Moreover, since
there could be an exponential number of local minimizers for the
$\ell_0$ regularized problem, there is no characterization  in
\cite{Lu:12} of the local minima at which the iterative hard
thresholding algorithm converges.
%These issues limit the application of this algorithm to general
%$\ell_0$ regularized optimization problems.
Further,  in \cite{LuZha:12}, penalty decomposition methods were
devised for both regularized and constrained formulations of sparse
nonconvex problems and convergence analysis was provided for these
algorithms. Analysis of sparsity constrained problems were provided
e.g. in \cite{Bec:12}, where the authors introduced several classes
of stationary points and developed greedy coordinate descent
algorithms converging to different classes of stationary points.
Coordinate descent methods are used frequently to solve sparse
optimization problems
\cite{Bec:12,LuXia:13,NecCli:13,NecNes:13,BanGha:08} since  they are
based on the strategy of updating one (block) coordinate of the
vector of variables per iteration using some index selection
procedure (e.g. cyclic, greedy or random). This often reduces
drastically the iteration complexity and memory requirements, making
these methods simple and scalable. There exist numerous papers
dealing with the convergence analysis of this type of methods: for
deterministic index selection see
\cite{HonWan:13,BecTet:13,LuoTse:92}, while for random index
selection see
\cite{LuXia:13,Nes:12,NecCli:13,Nec:13,NecPat:14,PatNec:14,RicTac:12}.

%%%%%%%%%%%%%%%%%%%%%%%%%%%%%%%%%%%%%%%%%%%%%%%%%%%%%%%%%%

\subsection{Main contribution}
\noindent In this paper we analyze a family of general random block
coordinate descent  iterative hard thresholding based  methods for
the minimization of $\ell_0$ regularized optimization problems, i.e.
the objective function is composed of a smooth convex function and
the $\ell_0$ regularization. The family of the algorithms  we
consider takes a very general form, consisting in the minimization
of a certain approximate version of the objective function  one
block variable at a time,  while fixing the rest of the block
variables. Such type of methods are particularly suited for solving
nonsmooth $\ell_0$ regularized problems since they solve an easy low
dimensional problem at each iteration, often in closed form. Our
family of methods covers particular cases such as random block
coordinate gradient descent and random proximal coordinate descent
methods. We analyze necessary optimality conditions for this
nonconvex $\ell_0$ regularized problem and devise a procedure for the  separation of
the set of local minima into restricted classes based on
approximation versions of the objective function. We provide a
unified analysis of the almost sure convergence for this family of
random  block coordinate descent algorithms and prove that, for each
approximation version, the limit points are  local minima from the
corresponding restricted class of local minimizers. Under the strong
convexity assumption, we prove linear convergence in probability for
our family of methods. We also provide numerical experiments which
show the superior behavior of our methods in comparison with the
usual iterative hard thresholding algorithm.

%%%%%%%%%%%%%%%%%%%%%%%%%%%%%%%%%%%%%%%%%%%%%%%

\subsection{Notations and preliminaries}
We consider the space $\rset^n$ composed  by column
vectors.  For $x,y \in \rset^n$ denote the scalar product by
$\langle x,y \rangle = x^T y$ and the  Euclidean norm by
$\|x\|=\sqrt{x^T x}$. We use the same notation $\langle \cdot,\cdot
\rangle$ ($\|\cdot\|$) for  scalar product (norm) in spaces of
different dimensions. For any matrix $A \in \rset^{m \times n}$ we
use $\sigma_{\min}(A)$ for the minimal
eigenvalue of matrix $A$. We use the notation $[n] = \{1,2,
\dots, n\}$ and $e = [1 \cdots 1]^T \in \rset^n$. \noindent In the
sequel, we consider the following decompositions of the variable
dimension and  of the $n \times n$ identity matrix:
\begin{equation*}
n = \sum\limits_{i=1}^N n_i, \qquad \qquad I_n= \left[ U_1 \dots U_N
\right], \qquad \qquad I_n= \left[ U_{(1)} \dots U_{(n)} \right],
\end{equation*}
where $U_i \in \rset^{n \times n_i}$ and $U_{(j)} \in
\rset^{n}$ for all  $i \in [N]$ and $j \in [n]$. If the
index set corresponding to block $i$ is given by $\mathcal{S}_i$,
then $\abs{\mathcal{S}_i} = n_i$. Given $x \in \rset^n$, then
for any $i \in [N]$ and $j \in [n]$, we denote:
\begin{align*}
 x_i &= U_i^T x \in \rset^{n_i}, \quad \quad \quad \quad \ \  \nabla_i f(x)= U_i^T \nabla f(x) \in \rset^{n_i},\\
 x_{(j)} &= U_{(j)}^T x \in \rset^{}, \quad \quad \quad \quad \ \ \ \nabla_{(j)} f(x)= U_{(j)}^T \nabla f(x) \in \rset^{}.
\end{align*}

%\noindent For any vector $M \in \rset^N_{++}$, we define the
%following pair of primal-dual norms:
%\begin{equation*}
%\norm{x}_M = \left(\sum\limits_{i=1}^N M_i\norm{x_i}^2\right)^{1/2}.
%\qquad \norm{y}_M^* = \left(\sum\limits_{i=1}^N
%\frac{1}{M_i}\norm{y_i}^2\right)^{1/2}.
%\end{equation*}

\noindent For any vector $x \in \rset^n$, the support of $x$ is
given by $\text{supp}(x)$, which denotes the set of indices corresponding
to the nonzero components of $x$. We denote $\bar{x} = \max\limits_{j
\in \text{supp}(x)} \abs{x_{(j)}}$ and $\underline{x} = \min\limits_{j \in
\text{supp}(x)} \abs{x_{(j)}}$. Additionally, we introduce the following
set of indices:
$$I(x) = \text{supp}(x) \cup \{j \in [n]: \ j \in \mathcal{S}_i, \ \lambda_i=0\}$$
and $I^c(x) = [n] \backslash I(x)$. Given two scalars $p\ge 1, r>0$
and $ x \in \rset^n$, the $p-$ball of radius $r$ and centered in $x$
is denoted by $\mathcal{B}_p(x,r) = \{y \in \rset^n: \; \norm{y-x}_p
< r \}$. Let $I \subseteq [n]$ and denote the subspace of all
vectors $x \in \rset^n$ satisfying $I(x) \subseteq I$ with $S_I$,
i.e. $S_I = \{x \in \rset^n: \; x_i=0 \quad \forall  i \notin I\}$.

\noindent We denote with $f^*$ the optimal value of the convex
problem $f^* = \min_{x \in \rset^n} f(x)$ and its optimal set with
$X^*_f=\left\{x \in \rset^n: \nabla f(x)=0 \right\}$.  In this paper
we consider the following assumption on function $f$:
\begin{assumption}\label{assump_grad_1}
The function $f$ has (block) coordinatewise Lipschitz continuous
gradient with constants $L_i>0$  for all $i \in [N]$, i.e. the
convex function $f$ satisfies the following inequality for all $i \in [N]$:
\begin{equation*}
  \norm{\nabla_i f(x+U_ih_i) - \nabla_i f(x)} \le L_i \norm{h_i}  \quad \forall x \in \rset^n, h_i \in \rset^{n_i}.
\end{equation*}
\end{assumption}
\noindent An immediate consequence of Assumption \ref{assump_grad_1}
is the following relation \cite{Nes:12}:
\begin{equation}\label{Lipschitz_gradient}
  f(x+U_ih_i) \le f(x) + \langle \nabla_i f(x), h_i\rangle +
  \frac{L_i}{2}\norm{h_i}^2 \quad \forall x \in \rset^n, h_i \in \rset^{n_i}.
\end{equation}
We denote with $\lambda=[\lambda_1 \cdots \lambda_N]^T \in \rset^N$,
$L=[L_1 \cdots L_N]^T $ and  $L_f$ the global Lipschitz constant
of the gradient $\nabla f(x)$. In the Euclidean settings, under
Assumption \ref{assump_grad_1}   a tight upper bound of the global
Lipschitz constant is  $L_f \le \sum_{i=1}^N L_i$ (see \cite[Lemma
2]{Nes:12}). Note that a global inequality based on $L_f$,
 similar to \eqref{Lipschitz_gradient},  can be also derived. Moreover,
we should remark that Assumption \ref{assump_grad_1} has been frequently considered in coordinate descent settings (see e.g. \cite{Nec:13,Nes:12,
NecNes:13,NecCli:13,NecPat:14,RicTac:12}).

%%%%%%%%%%%%%%%%%%%%%%%%%%%%%%%%%%%%%%%%%%%%%%%%%%%%%%%%%%%
%%%%%%%%%%%%%%%%%%%%%%%%%%%%%%%%%%%%%%%%%%%%%%%%%%%%%%%%%

\section{Characterization of local minima}
\noindent In this section we present the necessary optimality
conditions for problem \eqref{l0regular} and provide a detailed
description of local minimizers.  First, we establish necessary
optimality conditions satisfied by any local minimum. Then, we
separate the set of local minima into restricted classes around the
set of global minimizers. The next theorem provides conditions for
obtaining local minimizers of problem \eqref{l0regular}:

\begin{theorem}\label{lemmaaux}
If  Assumption \ref{assump_grad_1}   holds, then any $z \in \rset^n
\backslash \{0\}$ is a local minimizer of problem \eqref{l0regular}
on the ball $\mathcal{B}_{\infty}(z,r)$, with
$r=\min\left\{\underline{z},
 \frac{\underline{\lambda}}{\norm{\nabla f(z)}_1}\right\}$,
if and only if $z$ is a global minimizer of convex problem
$\min\limits_{x \in S_{I(z)}} f(x)$. Moreover, $0$ is a local
minimizer of problem \eqref{l0regular} on the ball
$\mathcal{B}_{\infty} \left(0,  \frac{\min_{i \in [N]}
\lambda_i}{\norm{\nabla f(z)}_1} \right)$ provided that $0 \not \in
X_f^*$, otherwise is a global minimizer for  \eqref{l0regular}.
\end{theorem}

\begin{proof}
For the first implication, we assume that $z$ is a local minimizer
of problem \eqref{l0regular} on the open ball
$\mathcal{B}_{\infty}(z,r)$, i.e. we have:
$$f(z) \le f(y)  \quad \forall y \in \mathcal{B}_{\infty}(z,r) \cap S_{I(z)}.$$

\noindent Based on Assumption \ref{assump_grad_1} it follows that
$f$ has also global Lipschitz continuous gradient, with constant
$L_f$, and thus we have:
$$f(z) \le f(y) \le f(z) + \langle \nabla f(z),y- z\rangle + \frac{L_f}{2}\norm{y- z}^2
 \quad \forall y \in \mathcal{B}_{\infty}(z,r)\cap S_{I(z)}.$$
Taking $\alpha = \min\{\frac{1}{L_f}, \frac{r}{\max\limits_{j \in
I(z)}\abs{\nabla_{(j)} f(z) }}\}$ and $y= z - \alpha \nabla_{I(z)}
f(z)$, we obtain:
$$0 \le \left(\frac{\alpha^2}{2L_f} - \frac{\alpha}{L_f}\right)\norm{\nabla_{I(z)} f(z)}^2 \le 0.$$
Therefore, we have $ \nabla_{I(z)} f(z) =0$, which means
that:
\begin{equation}\label{localmin}
z = \arg\min\limits_{x \in S_{I(z)}} f(x).
\end{equation}

\noindent For the second implication we first note that  for any $y,
d \in \rset^n$, with $ y \neq 0$ and $\norm{d}_{\infty} <
\underline{y} $,  we have:
\begin{equation}\label{positive}
    \abs{y_{(i)}+d_{(i)}} \ge \abs{y_{(i)}} - \abs{d_{(i)}} \ge \underline{y} - \norm{d}_{\infty} > 0  \quad \forall i \in \text{supp}(y).
\end{equation}
Clearly, for any $d \in \mathcal{B}_{\infty}(0,r) \backslash
S_{I(y)}$, with $r= \underline{y}$, we have:
\begin{equation*}
    \norm{y+d}_{0,\lambda} = \norm{y}_{0,\lambda} + \sum\limits_{i \in I^c(y) \cap \text{supp}(d)} \norm{d_{(i)}}_{0,\lambda} \ge \norm{y}_{0,\lambda} + \underline{\lambda}.
\end{equation*}

\noindent Let $d \in \mathcal{B}_{\infty}(0,r) \backslash S_{I(y)}$,
with $r = \min\left\{ \underline{y},
\frac{\underline{\lambda}}{\norm{\nabla f(y)}_1} \right\}$.  The
convexity of function $f$ and the Holder inequality lead to:
\begin{align}
\label{parti}
 F(y+d) &\ge f(y) + \langle \nabla f(y), d \rangle + \norm{y+d}_{0,\lambda} \nonumber \\
&\ge F(y) - \norm{\nabla f(y)}_1 \norm{d}_{\infty} +
\underline{\lambda} \ge F(y)  \quad \forall y \in \rset^n.
\end{align}
We now assume that $z$ satisfies \eqref{localmin}. For any $x \in
\mathcal{B}_{\infty}(z,r) \cap S_{I(z)}$ we have
$\norm{x-z}_{\infty} < \underline{z}$, which by \eqref{positive}
implies that $\abs{x_{(i)}} > 0$ whenever $\abs{z_{(i)}} > 0$.
Therefore, we get:
\begin{equation*}
 F(x) = f(x) + \norm{x}_{0,\lambda} \ge f(z) + \norm{z}_{0,\lambda} = F(z),
\end{equation*}
and combining with the inequality  \eqref{parti} leads to the second
implication. Furthermore,  if $0 \not \in X_f^*$, then $\nabla f(0)
\not =0$. Assuming that  $\min_{i \in [N]} \lambda_i>0$, then $F(x)
\geq f(0) + \langle \nabla f(0), x \rangle + \norm{x}_{0,\lambda}
\geq F(0) - \norm{\nabla f(0)}_1 \norm{x}_{\infty} + \min_{i \in
[N]} \lambda_i \geq F(0)$ for all $x \in \mathcal{B}_{\infty}
\left(0, \frac{\min_{i \in [N]} \lambda_i}{\norm{\nabla f(z)}_1}
\right)$. If $0 \in X_f^*$, then $\nabla f(0) =0$ and thus $F(x)
\geq f(0) + \langle \nabla f(0), z \rangle + \norm{x}_{0,\lambda}
\geq F(0)$ for all $x \in \rset^n$.
\end{proof}

\noindent From Theorem \ref{lemmaaux} we conclude that  any vector
$z \in \rset^n$ is a local minimizer of problem \eqref{l0regular} if
and only if the following equality holds:
$$ \nabla_{I(z)} f(z)=0.$$
We denote with $\mathcal{T}_f$ the set of all local minima of
problem \eqref{l0regular}, i.e.
\[  \mathcal{T}_f =\left\{  z \in \rset^n:\; \nabla_{I(z)} f(z)=0 \right \}, \]
and we call them \textit{basic local minimizers}.  It is not hard to
see that when the function $f$ is strongly convex, the number of
basic  local minima of problem \eqref{l0regular} is finite,
otherwise we might have an infinite number of basic local
minimizers.

%%%%%%%%%%%%%%%%%%%%%%%%%%%%%%%%%%%%%%%%%%%%%%%%%%%%%%%%%%%%%%%%%%
%%%%%%%%%%%%%%%%%%%%%%%%%%%%%%%%%%%%%%%%%%%%%%%%%%%%%%%%%%%%%%%%%%%

\subsection{Strong local minimizers}

\noindent In this section we introduce a family of strong local
minimizers of problem \eqref{l0regular} based on an approximation of
the function $f$.  It can be easily seen that finding a basic local
minimizer is a trivial procedure e.g.: $(a)$  if we choose some set
of indices $I \subseteq [n]$ such that $\{j \in [n]: j \in
\mathcal{S}_i, \lambda_i=0\} \subseteq I$, then from Theorem
\ref{lemmaaux}  the minimizer of the convex
 problem $\min_{x \in S_I} f(x)$ is a basic local minimizer for problem
\eqref{l0regular}; $(b)$ if we  minimize  the convex function $f$
w.r.t. all blocks $i$ satisfying $\lambda_i=0$, then from Theorem
\ref{lemmaaux}  we obtain again some basic local minimizer for
\eqref{l0regular}. This motivates us to introduce more restricted
classes of local minimizers. Thus, we first define an approximation
version of function $f$ satisfying certain assumptions. In
particular, given $i \in [N]$ and $x \in \rset^n$, the convex
function $u_i:\rset^{n_i} \to \rset^{}$ is an upper bound of
function $f(x+U_i(y_i-x_i))$ if it satisfies:
\begin{equation}
\label{u_upperapr}
 f(x+U_i(y_i-x_i)) \le u_i(y_i;x) \quad \forall y_i \in \rset^{n_i}.
\end{equation}

\noindent We additionally impose the following assumptions on each
function $u_i$.
\begin{assumption}\label{approximation}
The approximation function $u_i$ satisfies the  assumptions: \\
(i) The function $u_i(y_i;x)$ is strictly convex and differentiable
in the first argument, is continuous in the second argument and
satisfies $u_i(x_i;x) = f(x)$  for all
$x \in \rset^n$.\\
(ii) Its gradient in the first argument satisfies $\nabla u_i(x_i;x)
= \nabla_i f(x) \quad
\forall x \in \rset^n$. \\
(iii) For any $x \in \rset^n$, the  function $u_i(y_i;x)$ has
Lipschitz continuous gradient in the first argument with constant
$M_i > L_i$, i.e. there  exists  $M_i>L_i$ such that:
$$\norm{\nabla u_i(y_i;x) - \nabla u_i(z_i;x)} \le M_i \norm{y_i-z_i} \quad
  \forall y_i, z_i \in \rset^{n_i}.$$
(iv) There exists $\mu_i$ such that $0< \mu_i \le M_i - L_i$ and
$$u_i(y_i;x) \ge f(x+U_i(y_i-x_i)) + \frac{\mu_i}{2}\norm{y_i-x_i}^2 \quad \forall x \in \rset^n,
y_i \in \rset^{n_i}.$$
\end{assumption}
Note that a similar set of assumptions has been considered in
\cite{HonWan:13}, where the authors derived a general framework for
the block coordinate descent methods on composite convex problems.
Clearly,   Assumption \ref{approximation} $(iv)$ implies the upper
bound \eqref{u_upperapr} and in \cite{HonWan:13} this inequality  is replaced with
the assumption of strong convexity of $u_i$ in the first argument.

\noindent We now provide several examples of approximation versions
of the objective function $f$ which satisfy  Assumption
\ref{approximation}.

\begin{example}
\label{ex_3u} We now provide three examples of  approximation
versions for the function $f$. The reader can easily find many
other examples of approximations satisfying Assumption \ref{approximation}. \\
\text{1. Separable quadratic approximation}: given $M \in \rset^N$,
such that $M_i > L_i$ for all $i \in [N]$, we define the
approximation version
 $$u_i^q(y_i; x, M_i) = f(x) + \langle \nabla_i f(x), y_i-x_i\rangle +
 \frac{M_i}{2}\norm{y_i-x_i}^2.$$
It satisfies  Assumption \ref{approximation}, in particular
condition $(iv)$  holds for $\mu_i=M_i - L_i$. This type of
approximations was used by Nesterov  for deriving the random
coordinate gradient  descent method for solving smooth convex
problems \cite{Nes:12} and further extended to the composite convex
case in \cite{NecCli:13,RicTac:12}.

\vspace{3pt}

\noindent \text{2. General quadratic approximation}: given $H_i
\succeq 0$, such that $H_i \succ L_iI_{n_i}$  for all $i \in [N]$,
we define the approximation version
 $$u_i^Q(y_i;x,H_i) = f(x) + \langle \nabla_i f(x), y_i-x_i\rangle +
  \frac{1}{2}\langle y_i-x_i, H_i(y_i-x_i)\rangle. $$
 It satisfies  Assumption
\ref{approximation}, in particular condition $(iv)$  holds for
$\mu_i = \sigma_{\min}(H_i - L_i I_{n_i})$ (the smallest
eigenvalue).  This type of approximations was used by Luo, Yun  and
Tseng in deriving the greedy coordinate descent method based on the
Gauss-Southwell rule for solving composite convex problems
\cite{LuoTse:92,LuoTse:93,TseYun:09}.

\vspace{3pt}

\noindent \text{3. Exact approximation}: given $\beta \in \rset^N$,
such that $\beta_i > 0$ for all $i \in [N]$, we define the
approximation version
 $$u_i^{e}(y_i;x,\beta) = f(x+U_i(y_i-x_i)) + \frac{\beta_i}{2}\norm{y_i-x_i}^2.$$
It satisfies  Assumption \ref{approximation},  in particular
condition $(iv)$  holds for $\mu_i = \beta_i$. This type of
approximation functions was used especially in the nonconvex
settings \cite{GriSci:00,HonWan:13}.
 \end{example}

\vspace{3pt}

\noindent Based on each  approximation function $u_i$ satisfying
Assumption \ref{approximation}, we introduce a  class of restricted
local minimizers for our nonconvex optimization problem
\eqref{l0regular}.
\begin{definition}\label{local_min_general}
For any  set of approximation functions $u_i$ satisfying Assumption
\ref{approximation}, a vector $z$ is called an \textit{u-strong
local minimizer} for problem \eqref{l0regular} if it satisfies:
$$  F(z) \le \min\limits_{y_i \in \rset^{n_i}} u_i(y_i;z) + \norm{z+ U_i(y_i-z_i)}_{0,\lambda} \quad \forall i \in [N].$$
Moreover, we denote the set of strong local minima, corresponding to
the approximation functions $u_i$, with $\mathcal{L}_u$.
\end{definition}

\noindent It can be easily seen that  \[ \min_{y_i \in \rset^{n_i}}
u_i(y_i;z) + \norm{z+ U_i(y_i-z_i)}_{0,\lambda}
\overset{y_i=z_i}{\leq} u_i(z_i;z) + \norm{z}_{0,\lambda} = F(z) \]
and thus an u-strong local minimizer  $z \in \mathcal{L}_u$, has the
property that  each block $z_i$ is a fixed point of the operator
defined by the minimizers of the function $u_i(y_i;z) +
\lambda_i\norm{y_i}_0$, i.e. we have for all $i \in [N]$:
$$z_i = \arg\min\limits_{y_i \in \rset^{n_i}} u_i(y_i;z) +
\lambda_i\norm{y_i}_0.$$

\begin{theorem}
Let the set of approximation functions $u_i$ satisfy Assumption
\ref{approximation}, then any $u-$strong local minimizer  is a local
minimum of problem \eqref{l0regular}, i.e. the following inclusion
holds:
\[ \mathcal{L}_u \subseteq \mathcal{T}_f.  \]
\end{theorem}
\begin{proof}
From Definition \ref{local_min_general} and Assumption
\ref{approximation} we have:
\begin{align*}
 F(z) & \le \min\limits_{y_i \in \rset^{n_i}} u_i(y_i;z) + \norm{z+ U_i(y_i-z_i)}_{0,\lambda}\\
&\le \min\limits_{y_i \in \rset^{n_i}} u_i(z_i;z) + \langle \nabla u_i (z_i;z), y_i-z_i \rangle + \frac{M_i}{2}\norm{y_i-z_i}^2+ \norm{z+ U_i(y_i-z_i)}_{0,\lambda}\\
&=  \min\limits_{y_i \in \rset^{n_i}} F(z) + \langle \nabla_i f(z), y_i-z_i \rangle + \frac{M_i}{2}\norm{y_i-z_i}^2+ \lambda_i(\norm{y_i}_0-\norm{z_i}_{0})\\
&\le F(z) + \langle \nabla_i f(z), h_i \rangle +
\frac{M_i}{2}\norm{h_i}^2+
\lambda_i(\norm{z_i+h_i}_0-\norm{z_i}_{0})
\end{align*}
for all $h_i \in \rset^{n_i}$ and  $i \in [N]$. Choosing now $h_i$
as follows:
\[ h_{i} = -\frac{1}{M_i}U_{(j)}\nabla_{(j)} f(z) \;\;\;  \text{for some} \;\; j \in I(z)
\cap \mathcal{S}_i,\] we have from the definition of $I(z)$ that
\[ \lambda_i(\norm{z_i+h_i}_0-\norm{z_i}_{0}) \leq 0 \]
and thus $0 \leq -\frac{1}{2M_i} \|\nabla_{(j)} f(z) \|^2$ or
equivalently $\nabla_{(j)} f(z) = 0$. Since this holds for any $j
\in I(z) \cap \mathcal{S}_i$, it follows that $z$ satisfies
$\nabla_{I(z)} f(z) = 0$. Using now Theorem \ref{lemmaaux} we obtain
our statement.
\end{proof}

\noindent For the three approximation versions given in Example
\ref{ex_3u} we obtain explicit expressions for the corresponding
u-strong local minimizers. In particular, for some $M \in
\rset^N_{++}$ and $i \in [N]$, if we consider the previous separable
quadratic  approximation $u_i^q(y_i;x,M_i)$, then any strong local
minimizer $z \in \mathcal{L}_{u^q}$ satisfies the following
relations:
\begin{enumerate}
\item[\textit{(i)}] $\nabla_{I(z)} f(z) = 0$ and additionally
\item[\textit{(ii)}]
$\begin{cases} \abs{\nabla_{(j)} f(z)} \le \sqrt{2\lambda_{i} M_{i} }, &\text{if} \ z_{(j)}=0 \\
\abs{z_{(j)}} \ge \sqrt{\frac{2\lambda_{i}}{M_{i}}}, &\text{if} \
z_{(j)} \neq 0, \quad \forall i \in [N]$ and $j \in \mathcal{S}_i.
\end{cases}$
\end{enumerate}
 The relations given in $(ii)$ can be derived  based on  the separable
 structure of the approximation $u_i^q(y_i;x,M_i)$ and of the
 quasinorm $\|\cdot\|_0$ using similar arguments as in Lemma 3.2 from \cite{Lu:12}. For completeness, we present the main steps in the derivation. First, it is clear that any $z \in \mathcal{L}_{u^q}$ satisfies:
\begin{equation}\label{u^q_argmin}
z_{(j)} = \arg\min_{y_{(j)} \in \rset^{}} \nabla_{(j)} f(z) (y_{(j)}-z_{(j)}) + \frac{M_i}{2}\abs{y_{(j)}-z_{(j)}}^2
+ \lambda_i \norm{y_{(j)}}_0
\end{equation}
for all $j \in \mathcal{S}_i $  and $ i\in [N]$.  On the other hand since the optimum point in the previous optimization problems can be $0$ or different from $0$, we have:
\begin{align*}
&\min_{y_{(j)} \in \rset^{}} \nabla_{(j)} f(z) (y_{(j)}-z_{(j)}) + \frac{M_i}{2}\abs{y_{(j)}-z_{(j)}}^2 + \lambda_i \norm{y_{(j)}}_0 \\
&=\min \left\{ \frac{M_i}{2}\abs{z_{(j)}- \frac{1}{M_i}\nabla_{(j)} f(z)}^2 - \frac{1}{2M_i}\abs{\nabla_{(j)} f(z)}^2, \lambda_i - \frac{1}{2M_i}\abs{\nabla_{(j)} f(z)}^2 \right\}.
\end{align*}
If $z_{(j)}=0$, then from fixed point relation of  problem \eqref{u^q_argmin} and the expression for its  optimal value we have $\frac{M_i}{2}\abs{z_{(j)}- \frac{1}{M_i}\nabla_{(j)} f(z)}^2 - \frac{1}{2M_i}\abs{\nabla_{(j)} f(z)}^2 \leq \lambda_i - \frac{1}{2M_i}\abs{\nabla_{(j)} f(z)}^2$ and thus
$\abs{\nabla_{(j)} f(z)} \le \sqrt{2\lambda_i M_i}$. Otherwise,  we have $j \in I(z)$  such that  from Theorem \ref{lemmaaux} we have  $\nabla_{(j)} f(z) =0$ and combining with $\frac{M_i}{2}\abs{z_{(j)}- \frac{1}{M_i}\nabla_{(j)} f(z)}^2 - \frac{1}{2M_i}\abs{\nabla_{(j)} f(z)}^2 \geq \lambda_i - \frac{1}{2M_i}\abs{\nabla_{(j)} f(z)}^2$ leads to $\abs{z_{(j)}} \ge \sqrt{\frac{2\lambda_i}{M_i}}$. Similar derivations as above  can be derived for the general quadratic approximations $u_i^Q(y_i;x,H_i)$ provided that $H_i$ is diagonal matrix. For general matrices $H_i$, the corresponding strong local minimizers are fixed points of small $\ell_0$ regularized quadratic problems of dimensions $n_i$.

\noindent Finally, for some $\beta \in \rset^N_{++}$ and $i \in
[N]$, considering the  exact approximation $u_i^{e}(y_i;x,\beta_i)$
we obtain that any corresponding strong local minimizer $z \in
\mathcal{L}_{u^e}$ satisfies:
\begin{equation*}
 z_i = \arg \min\limits_{h_i \in \rset^{n_i}} F(z+ U_ih_i) + \frac{\beta_{i}}{2}\norm{h_i}^2 \quad \forall i \in [N].
\end{equation*}

\begin{theorem}\label{inclusions}
Let  Assumption \ref{assump_grad_1} hold and $u^1, u^2$ be two approximation functions satisfying Assumption \ref{approximation}.
Additionally, let
$$u^1(y_i;x) \le u^2(y_i;x), \quad \forall y_i \in \rset^{n_i}, x \in \rset^n, i \in [N].$$
Then the following inclusions are valid:
$$ \mathcal{X}^* \subseteq \mathcal{L}_{u^1} \subseteq \mathcal{L}_{u^2} \subseteq \mathcal{T}_f.$$
\end{theorem}

\begin{proof}
Assume $z \in \mathcal{X}^*$, i.e. it is a global minimizer of our
original nonconvex problem \eqref{l0regular}. Then, we have:
\begin{align*}
F(z) &\le \min\limits_{y_i \in \rset^{n_i}} F(z + U_i(y_i-z_i)) \\
& = \min\limits_{y_i \in \rset^{n_i}} f(z+U_i(y_i-z_i)) + \lambda_i \norm{y_i}_0 + \sum\limits_{j \neq i}\lambda_j \norm{z_j}_0\\
&\le \min\limits_{y_i \in \rset^{n_i}} u_i^1(y_i;z)
+\norm{z+U_i(y_i-z_i)}_{0,\lambda} \quad \forall i\in [N],
\end{align*}
and thus $z \in \mathcal{L}_{u^1}$, i.e.  we  proved that
$\mathcal{X}^* \subseteq \mathcal{L}_{u^1}$. Therefore, any class of
$u$-strong local minimizers contains the global minima of problem
\eqref{l0regular}.

\noindent Further, let us  take  $z \in \mathcal{L}_{u^1}$. Using
Definition \eqref{local_min_general} and defining  \[ t_i =
\arg\min\limits_{y_i \in \rset^{n_i}} u_i^2(y_i;z) +
\norm{z+U_i(y_i-z_i)}_{0,\lambda}, \] we get:
\begin{align*}
 F(z) &\le \min\limits_{y_i \in \rset^{n_i}} u_i^1(y_i;z) + \norm{z+U_i(y_i-z_i)}_{0,\lambda}\\
&\le u_i^1(t_i;z) + \norm{z+U_i(t_i-z_i)}_{0,\lambda}\\
&\le u_i^2(t_i;z) + \norm{z+U_i(t_i-z_i)}_{0,\lambda}\\
&= \min\limits_{y_i \in \rset^{n_i}} u_i^2 (y_i;z) +
\norm{z+U_i(y_i-z_i)}_{0,\lambda}.
\end{align*}
This shows that $z \in \mathcal{L}_{u^2}$ and thus
$\mathcal{L}_{u^1} \subseteq \mathcal{L}_{u^2}$.
\end{proof}

\noindent Note that if the following inequalities hold \[ (L_i +
\beta_i) I_{n_i} \preceq H_i \preceq M_i I_{n_i} \quad \forall i \in
[N],\] using the Lipschitz gradient relation
\eqref{Lipschitz_gradient}, we obtain that
$$u_i^{e}(y_i;x,\beta_i) \le u_i^{Q}(y_i;x,H_i) \leq  u_i^{q}(y_i;x,M_i) \quad \forall x \in \rset^n,
y_i \in \rset^{n_i}.$$ \noindent Therefore, from Theorem
\ref{inclusions} we observe that $u^{q} \ (u^{Q})$-strong local
minimizers for problem \eqref{l0regular} are included in the class
of all basic local minimizers $\mathcal{T}_f$. Thus, designing an
algorithm which converges to a local minimum from
$\mathcal{L}_{u^q}$ ($\mathcal{L}_{u^Q}$)  will be of interest.
Moreover,  $u^e$-strong local minimizers for problem
\eqref{l0regular} are included in the class of all $u^{q} \
(u^Q)$-strong local minimizers. Thus, designing an algorithm which
converges to a local minimum from $\mathcal{L}_{u^{e}}$ will be of
interest. To illustrate the relationships between the previously
defined classes of restricted local minima and see how much they are
related to global minima of \eqref{l0regular}, let us consider an
example.

\begin{example}
\label{classes_points} We consider the least square settings $f(x) =
\norm{Ax-b}^2$, where $A \in \rset^{m \times n}$ and $b \in \rset^m$
satisfying:
\begin{equation*}
 A = \begin{bmatrix}
 1 & \alpha_1 &\cdots &\alpha_1^n \\
 1 &\alpha_2 &\cdots &\alpha_2^n \\
 1 &\alpha_3 &\cdots &\alpha_3^n \\
 1 &\alpha_4 &\cdots &\alpha_4^n \\
\end{bmatrix}
+ \left[p I_4 \quad O_{4,n-4}  \right], \qquad b=q e,
\end{equation*}
with $e \in \rset^{4}$ the vector having all entries $1$. We choose
the following parameter values: $\alpha = [1 \ 1.1 \ 1.2 \ 1.3]^T,
n=7, p=3.3, q=25, \lambda=1$ and $\beta_i=0.0001$ for all $i\in
[n]$. We further consider the scalar case, i.e. $n_i =1$ for all
$i$. In this case we have that $u_i^q=u_i^Q$, i.e. the separable and
general quadratic approximation versions coincide. The results are
given in Table \ref{local_min}. From $128$ possible local minima, we
found $19$ local minimizers in $\mathcal{L}_{u^{q}}$ given by
$u^q_i(y_i;x,L_f)$, and only $6$ local minimizers in
$\mathcal{L}_{u^q}$ given by $u^q_i(y_i;x,L_i)$. Moreover, the class
of $u^{e}$-strong local minima $\mathcal{L}_{u^{e}}$ given by
$u^e_i(y_i;x,\beta_i)$ contains only one vector which is also the
global optimum of problem \eqref{l0regular}, i.e. in this case
$\mathcal{L}_{u^{e}} = \mathcal{X}^*$. From Table \ref{local_min} we
can clearly see that the newly introduced classes of local
minimizers  are much more restricted (in the sense of having small
number of elements,  close to that of  the set of global minimizers)
than the class of basic local minimizers that is much larger.

\setlength{\extrarowheight}{5pt}
\renewcommand{\tabcolsep}{4pt}
\begin{center}
\begin{table}[ht]
\begin{center}\caption{Strong local minima distribution on a least square example.}
\label{local_min}
\begin{tabular}{|c|c|c|c|c|}
\hline \textbf{Class of local minima} & $\mathcal{T}_f$ &
$\overset{\mathcal{L}_{u^{q}}}{u^q_i(y_i;x,L_f)}$
& $\overset{\mathcal{L}_{u^{q}}}{u^q_i(y_i;x,L_i)}$ & $\overset{\mathcal{L}_{u^{e}}}{u^e_i(y_i;x,\beta_i)}$  \\
\hline
\textbf{Number of local minima} & 128 & 19 & 6 & 1  \\
\hline
\end{tabular}
\end{center}
\end{table}
\end{center}
\end{example}

%%%%%%%%%%%%%%%%%%%%%%%%%%%%%%%%%%%%%%%%%%%%%%%%%%%%%%%%%%%%%%%%%%%%%%%%%%%%%%%%%%%%%%%%%%%%%%%%%%%%%%%%%%%%%%%%%%%%%%%%%%%%%%%%%%%%%%5
%%%%%%%%%%%%%%%%%%%%%%%%%%%%%%%%%%%%%%%%%%%%%%%%%%%%%%%%%%%%%%%%%%%%%%%%%%%%%%%%%%%%%%%%%%%%%%%%%%%%%%%%%%%%%%%%%%%%%%%%%%%%%%%%%%%%%%5

\section{Random coordinate descent type methods}

In this section we present a family of random block coordinate descent methods
suitable for solving the class of problems \eqref{l0regular}. The family of the algorithms we consider takes a very general form, consisting in the minimization of a certain approximate
version of the objective function one block variable at a time, while fixing the
rest of the block variables. Thus,  these algorithms are a combination between  an iterative hard thresholding scheme and a general random coordinate descent method and they are particularly suited for solving nonsmooth $\ell_0$ regularized problems since they solve an easy low dimensional problem at each iteration, often in closed form. Our family of methods covers particular cases such as random block coordinate gradient descent and random proximal coordinate descent methods.

\noindent Let $x \in \rset^n$ and  $i \in
[N]$. Then, we introduce the following \textit{thresholding map} for a given approximation version $u$ satisfying Assumption \ref{approximation}:
\begin{align*}
 T^{u}_i(x) &=
 \arg\min\limits_{y_i \in \rset^{n_i}} u_i(y_i;x) + \lambda_i\norm{y_i}_{0}.
\end{align*}

\noindent In order to find a  local minimizer of problem
\eqref{l0regular}, we introduce the family of  \textit{random block coordinate
descent  iterative hard thresholding} (RCD-IHT) methods, whose
iteration  is described as follows:
\begin{algorithm}{{\bf RCD-IHT}}
\begin{itemize}
\item[1.]  Choose $ x^0 \in \rset^n$   and approximation version $u$ satisfying Assumption \ref{approximation}. For  $k \ge 0$ do:

\item[2.]  Choose a (block) coordinate  $i_k \in [N]$ with uniform probability

\item[3.]  Set  $x^{k+1}_{i_k} = T^{u}_{i_k}(x^k)$ and $x^{k+1}_i=x^k_i \;\; \forall i \neq i_k$.
\end{itemize}
\end{algorithm}

\vspace{3pt}

\noindent Note that our algorithm is  directly dependent on the
choice of approximation $u$ and the computation of the operator
$T^{u}_{i}(x)$ is in general easy, sometimes even in closed form.
For example, when $u_i(y_i;x) = u_i^q(y_i;x,M_i)$ and $\nabla_{i_k}
f(x^k)$ is available, we can easily compute the closed form solution
of $T^{u}_{i_k}(x^k)$ as in the iterative hard thresholding schemes
\cite{Lu:12}. Indeed, if  we define $\Delta^i(x) \in \rset^{n_i}$ as
follows:
\begin{align}
\label{deltaq} (\Delta^i(x))_{(j)} = \frac{M_i}{2} \abs{x_{(j)}-
(1/M_i)\nabla_{(j)} f(x)}^2,
\end{align}
then  the iteration of (RCD-IHT) method becomes:
\begin{equation*}
x^{k+1}_{(j)} =
\begin{cases}
x^k_{(j)} - \frac{1}{M_{i_k}}\nabla_{(j)}f(x^k), & \text{if} \quad
 (\Delta^{i_k}(x^k))_{(j)}\ge \lambda_{i_k}\\
0, &\text{if} \quad (\Delta^{i_k}(x^k))_{(j)}\le \lambda_{i_k},
\end{cases}
\end{equation*}
for all $j \in \mathcal{S}_{i_k}$.  Note that if at some iteration
$\lambda_{i_k}=0$, then the iteration of algorithm (RCD-IHT) is
identical with the iteration of the usual \textit{random block
coordinate gradient  descent method} \cite{NecCli:13,Nes:12}.
Further,  our algorithm  has, in this case, similarities with the iterative hard
thresholding algorithm (IHTA) analyzed  in \cite{Lu:12}. For
completeness, we also present the algorithm (IHTA).

\begin{algorithm}{{IHTA}} \cite{Lu:12}
\begin{itemize}
\item[1.]  Choose $ M_f > L_f$. For  $k \ge 0$ do:

\item[2.] $x^{k+1} = \arg\min_{y \in \rset^{n}} f(x^k) +
\langle \nabla f(x^k), y - x^k \rangle + \frac{M_f}{2}\norm{ y -
x^k}^2 + \norm{y}_{0,\lambda} $,
\end{itemize}
\end{algorithm}
or equivalently for each component we have the update:
\begin{equation*}
x^{k+1}_{(j)} =
\begin{cases}
x^k_{(j)} - \frac{1}{M_{f}}\nabla_{(j)}f(x^k), & \text{if} \quad
\frac{M_f}{2} \abs{x^k_{(j)} - \frac{1}{M_{f}}
\nabla_{(j)}f(x^k)}^2 \ge \lambda_{i}\\
    0, &\text{if} \quad \frac{M_f}{2} \abs{x^k_{(j)} - \frac{1}{M_{f}}\nabla_{(j)}f(x^k)}^2 \le
     \lambda_{i},
\end{cases}
\end{equation*}
for all $j \in \mathcal{S}_i$ and $i \in [N].$ Note that the
arithmetic complexity of computing the next iterate $x^{k+1}$ in
(RCD-IHT), once $\nabla_{i_k}f(x^k)$ is known, is of order
$\mathcal{O}(n_{i_k})$, which is much lower than the arithmetic
complexity per iteration $\mathcal{O}(n)$ of (IHTA) for $N >> 1$,
that additionally requires the computation of full gradient $\nabla
f(x^k)$.  Similar derivations as above  can be derived for the
general quadratic approximations $u_i^Q(y_i;x,H_i)$ provided that
$H_i$ is diagonal matrix. For general matrices $H_i$, the
corresponding algorithm requires solving small $\ell_0$ regularized
quadratic problems of dimensions $n_i$.

\noindent Finally, in the particular case when we consider the exact
approximation  $u_i(y_i;x)=u_i^{e}(y_i;x,\beta_i)$, at  each
iteration of our algorithm we need to perform an exact minimization
of the objective function $f$ w.r.t. one randomly chosen (block)
coordinate. If $\lambda_{i_k} = 0$, then the iteration of algorithm
(RCD-IHT) requires solving a small dimensional subproblem with a
strongly convex objective function  as in the classical
\textit{proximal block  coordinate descent method} \cite{HonWan:13}.
In the case when $\lambda_{i_k}>0$ and $n_i > 1$, this subproblem is
nonconvex and usually hard to solve. However, for certain particular
cases of the function $f$ and $n_i = 1$ (i.e. scalar case $n=N$), we
can easily compute the  solution of the small dimensional subproblem
in algorithm (RCD-IHT). Indeed, for $x \in \rset^n$ let us define:
\begin{align}
\label{deltae} v^i(x) & = x + U_i h_i(x), \ \text{where} \ h_i(x) =
\arg\min\limits_{h_i \in \rset^{}}
f(x + U_ih_i) + \frac{\beta_i}{2}\norm{h_i}^2  \nonumber\\
\Delta^i(x) &= f(x-U_ix_i) + \frac{\beta_i}{2}\norm{x_i}^2 -
f(v^{i}(x)) - \frac{\beta_i}{2}\norm{(v^{i}(x))_i - x_i}^2 \quad
\forall i \in [n].
\end{align}
Then, it can be  seen that the iteration of (RCD-IHT) in the scalar
case for the exact approximation $u_i^{e}(y_i;x,\beta_i)$  has the
following form:
\begin{equation*}
x^{k+1}_{i_k} =
\begin{cases}
(v^{i_k}(x^k))_{i_k}, \ &\text{if} \ \Delta^{i_k}(x^k) \ge \lambda_{i_k} \\
 0 , \ &\text{if} \ \Delta^{i_k}(x^k) \le \lambda_{i_k}.
\end{cases}
\end{equation*}
In general, if the function $f$ satisfies Assumption
\ref{assump_grad_1}, computing $v^{i_k}(x^k)$ at each iteration of
(RCD-IHT) requires the  minimization of an unidimensional convex
smooth function, which can be  efficiently performed using
unidimensional search algorithms. Let us analyze the least squares
settings in order to highlight the simplicity of the iteration of
algorithm (RCD-IHT)  in the scalar case for the approximation
$u_i^{e}(y_i;x,\beta_i)$.

\begin{example}
Let $A \in \rset^{m \times n}, b \in \rset^m$ and $f(x) =
\frac{1}{2}\norm{Ax-b}^2$. In this case (recall that we consider
$n_i=1$ for all $i$) we have the following expression for
$\Delta^{i}(x)$:
$$\Delta^{i}(x) =\frac{1}{2}\norm{r-A_{i}x_{i}}^2+\frac{\beta_{i}}{2}\norm{x_{i}}^2
- \frac{1}{2}\left\lVert
r\left(I_m-\frac{A_{i}A_{i}^T}{\norm{A_{i}}^2+\beta_{i}}\right)\right\rVert^2-\frac{\beta_{i}}{2}\left
\lVert \frac{A_{i}^Tr}{\norm{A_{i}}^2+\beta_{i}}\right\rVert^2,$$
where $r = Ax-b$. Under these circumstances, the iteration of
(RCD-IHT) has the following closed  form expression:
\begin{equation}
x^{k+1}_{i_k}=
\begin{cases}
x^k_{i_k} - \frac{A_{i_k}^Tr^k}{\norm{A_{i_k}}^2+\beta_{i_k}}, \
&\text{if} \
\Delta^{i_k}(x^k) \ge \lambda_{i_k}   \\
 0 , \ &\text{if} \ \Delta^{i_k}(x^k) \le \lambda_{i_k}.
\end{cases}
\end{equation}
\end{example}

\noindent In the sequel  we use the following notations for the
entire history of index choices, the expected value of objective
function $f$ w.r.t. the entire history and for the support of the
sequence $x^k$:
\begin{equation*}
 \xi^k = \{i_0, \dots, i_{k-1} \}, \qquad f^k=\mathbb{E}[f(x^k)], \qquad I^k=I(x^k).
\end{equation*}

\noindent Due to the randomness of algorithm (RCD-IHT),
at any iteration $k$ with $\lambda_{i_k} > 0$, the
sequence $I^k$ changes if one of the following situations holds
for some $j \in \mathcal{S}_{i_k}$:
\begin{align*}
(i)& \ x^k_{(j)} = 0 \ \text{and} \ (T^u_{i_k}(x^k))_{(j)} \neq 0 \\
(ii)& \ x^k_{(j)} \neq 0 \ \text{and} \ (T^u_{i_k}(x^k))_{(j)} = 0.
\end{align*}

\noindent In other terms, at a given moment $k$ with
$\lambda_{i_k}>0$, we expect  no change in the sequence $I^k$ of
algorithm (RCD-IHT) if there is no index $j \in \mathcal{S}_{i_k}$
satisfying the above corresponding set of relations $(i)$ and
$(ii)$. We define the  notion of \textit{change of $I^k$ in
expectation} at iteration $k$, for algorithm (RCD-IHT) as follows:
let $x^k$ be the sequence generated by (RCD-IHT), then the sequence
$I^k=I(x^k)$ changes in expectation if the following situation
occurs:
\begin{equation}\label{expectation}
\mathbb{E}[\abs{I^{k+1} \setminus I^k} + \abs{I^k \setminus I^{k+1}}
\ | \  x^k] > 0,
\end{equation}
which implies (recall that we consider uniform probabilities for the
index selection):
\begin{align*}
\mathbb{P}\left(\abs{I^{k+1} \setminus I^k} + \abs{I^k \setminus
I^{k+1}} > 0 \ | \  x^k \right) \ge \frac{1}{N}.
\end{align*}

\noindent In the next section we show that there is a finite  number
of changes of  $I^k$  in expectation generated by algorithm
(RCD-IHT) and then, we prove  global convergence of this algorithm,
in particular we show that the limit points of the generated
sequence  converges to strong local minima from the class of points
$\mathcal{L}_{u}$.

%%%%%%%%%%%%%%%%%%%%%%%%%%%%%%%%%%%%%%%%%%%%%%%%%%%%%%%%%%%%%%%%%%%%%%%%%%%%%%%%%%%%%%%%%%%%%%%%%%%%%%%%%%%%%%%
%%%%%%%%%%%%%%%%%%%%%%%%%%%%%%%%%%%%%%%%%%%%%%%%%%%%%%%%%%%%%%%%%%%%%%%%%%%%%%%%%%%%%%%%%%%%%%%%%%%%%%%%%%%%%%%

\section{Global convergence analysis}
\noindent In this section we analyze the descent properties of the
previously introduced family of coordinate descent algorithms under
Assumptions \ref{assump_grad_1} and \ref{approximation}. Based on
these properties, we establish the nature of the limit points of the
sequence generated by Algorithm (RCD-IHT). In particular, we derive
that any accumulation point of this sequence  is almost surely a
local minimum which belongs to the class $\mathcal{L}_u$. Note that
the classical results for any iterative algorithm used for solving
general  nonconvex problems state global convergence to stationary
points, while for the $\ell_0$ regularized  nonconvex and NP-hard
problem \eqref{l0regular} we show that our family of algorithms have
the property that  the  generated  sequences converge to strong
local minima.

\noindent In order to prove almost sure convergence results for our
family of algorithms, we use the following supermartingale
convergence lemma of Robbins and Siegmund (see e.g.
\cite{PatNec:14}):
\begin{lemma}
\label{mart} Let $v_k, u_k$ and $\alpha_k$ be three sequences of
nonnegative random variables satisfying the following conditions:
\[ \mathbb{E}[v_{k+1} | {\cal F}_k] \leq (1+\alpha_k) v_k - u_k
\;\; \forall k \geq 0 \;\; \text{a.s.} \;\;\;  \text{and} \;\;\;
\sum_{k=0}^\infty \alpha_k < \infty \;\;  \text{a.s.}, \]  where
${\cal F}_k$ denotes the collections $v_0, \dots, v_k, u_0, \dots,
u_k$, $\alpha_0, \dots, \alpha_k$. Then, we have $\lim_{k \to
\infty} v_k = v$ for a random variable $v \geq 0$ a.s. \;  and \;
$\sum_{k=0}^\infty u_k < \infty$ a.s.
\end{lemma}

\noindent Further, we analyze the convergence properties of
algorithm (RCD-IHT). First, we derive a descent inequality for this
algorithm.

\begin{lemma}
\label{descent_ramiht}
Let $x^k$ be the sequence generated by (RCD-IHT) algorithm. Under
Assumptions \ref{assump_grad_1}  and \ref{approximation} the
following descent inequality holds:
\begin{align}
 \label{decrease}
\mathbb{E}[F(x^{k+1})\;|\; x^k] \le F(x^k) -
\mathbb{E}\left[\frac{\mu_{i_k}}{2}\norm{x^{k+1}-x^k}^2 \;|\; x^k
\right].
\end{align}
\end{lemma}
\begin{proof}
From Assumption \ref{approximation}  we have:
\begin{align*}
F(x^{k+1}) + \frac{\mu_{i_k}}{2}\norm{x^{k+1}_{i_k}-x_{i_k}^k}^2 &\le
u_{i_k}(x^{k+1}_{i_k}, x^k) + \norm{x^{k+1}}_{0,\lambda} \\
&\le u_{i_k}(x^{k}_{i_k}, x^k) + \norm{x^{k}}_{0,\lambda} \\
&\le f(x^k) + \norm{x^{k}}_{0,\lambda} = F(x^k).
\end{align*}
In conclusion, our family of algorithms belong to the class of
descent methods:
\begin{align}
\label{decrease_iter}
F(x^{k+1}) & \le F(x^k) -
\frac{\mu_{i_k}}{2}\norm{x^{k+1}_{i_k}-x_{i_k}^k}^2.
\end{align}
Taking expectation w.r.t. $i_k$ we get our descent inequality.
\end{proof}

\noindent We now prove the global convergence of the sequence
generated by algorithm (RCD-IHT) to local minima which belongs to
the restricted set of local minimizers  $\mathcal{L}_u$.

\begin{theorem}
\label{convergence_rpamiht} Let $x^k$ be the sequence generated by
algorithm (RCD-IHT). Under Assumptions \ref{assump_grad_1} and
\ref{approximation} the following statements hold:

\noindent $(i)$ There exists a scalar $\tilde{F}$ such that:
$$ \lim\limits_{k \to \infty} F(x^{k})=  \tilde{F}  \ a.s. \quad
\text{and} \quad \lim\limits_{k \to \infty} \norm{x^{k+1}-x^k} = 0 \
a.s.$$

\noindent $(ii)$ At each change of sequence $I^k$ in expectation we
have the following relation:
$$\mathbb{E}\left[\frac{\mu_{i_k}}{2}\norm{x^{k+1}-x^k}^2 \;|\; x^k \right] \ge \delta,$$
where $\delta= \frac{1}{N}\min\left\{\min\limits_{i \in [N]:
\lambda_i>0} \frac{\mu_i\lambda_i }{M_i}, \min\limits_{i \in [N], j
\in {\mathcal S}_i \cap \text{supp}(x^0)} \frac{ \mu_i}{2}
|x^0_{(j)}|^2 \right\} > 0.$

\noindent $(iii)$ The sequence $I^k$ changes a finite number of
times as $k \to \infty$ almost surely. The sequence $\norm{x^k}_0$
converges to some $\norm{x^*}_0$ almost surely. Furthermore, any
limit point of the sequence $x^k$ belongs to the class of strong
local minimizers $\mathcal{L}_u$ almost surely.
\end{theorem}

\begin{proof}
\noindent $(i)$  From the descent inequality given in Lemma
\eqref{descent_ramiht} and Lemma \ref{mart} we  have that there
exists a scalar $\tilde{F}$ such that  $\lim_{k \to \infty}
F(x^{k})=  \tilde{F}$ almost sure. Consequently, we also have
$\lim_{k \to \infty} F(x^k) - F(x^{k+1}) = 0$ almost sure and since
our method is  of descent type, then from \eqref{decrease_iter} we
get $\frac{\mu_{i_k}}{2}\norm{x^{k+1} - x^k}^2 \le F(x^k) -
F(x^{k+1})$, which leads to $\lim_{k \to \infty} \norm{x^{k+1}-x^k}
= 0$ almost sure.

\noindent $(ii)$ For simplicity of the notation we denote $x^{+} =
x^{k+1}, x = x ^k $ and $ i=i_k$. First, we show that any nonzero
component of the sequence generated by (RCD-IHT) is bounded below by
a positive constant. Let $x \in \rset^n$ and $i \in [N]$. From
definition of $T^{u}_i(x)$,  for any $j \in
\text{supp}(T^{u}_i(x))$, the  $j$th component of the minimizer
$T^{u}_i(x)$ of the function $u_i(y_i;x) + \lambda_i\norm{y_i}_{0}$
is denoted  $(T^{u}_i(x))_{(j)}$. Let us define  $y^+ = x +
U_i(T^{u}_i(x)-x_i)$. Then,  for any $j \in \text{supp}(T^{u}_i(x))$
the following optimality condition holds:
\begin{align}
\label{nablaui} \nabla_{(j)} u_i(y^+_i;x)=0.
\end{align}

\noindent On the other hand, given $j \in \text{supp}(T^{u}_i(x))$, from the definition of $T^{u}_i(x)$ we get:
\begin{align*}
u_i(y^+_i;x) + \lambda_i \norm{y^+_i}_0 \le
u_i(y^{+}_i-U_{(j)}y^+_{(j)}; x) + \lambda_i
\norm{y^{+}_i-U_{(j)}y^{+}_{(j)}}_0.
\end{align*}
Subtracting $\lambda_i \norm{y^{+}_i-U_{(j)}y^{+}_{(j)}}_{0}$ from
both sides, leads to:
\begin{equation} \label{ineq_iter}
u_i(y^+_i;x) + \lambda_i \le  u_i(y^{+}_i-U_{(j)}y^+_{(j)}; x) .
\end{equation}
Further, if we apply the Lipschitz gradient relation given in
Assumption \ref{approximation} $(iii)$ in the right hand side  and
use the optimality conditions for the unconstrained problem solved
at each iteration, we get:
\begin{align*}
u_i(y^{+}_i-U_{(j)}y^+_{(j)}; x) &\le u_i(y^{+}_i;x) - \langle \nabla_{(j)} u_i(y^{+}_i;x),y^{+}_{(j)}\rangle\
 + \frac{M_{i}}{2} |y^{+}_{(j)}|^2\\
& \overset{\eqref{nablaui}}{=} u_i(y^{+}_i;x) + \frac{M_{i}}{2}
|y^{+}_{(j)}|^2.
\end{align*}
Combining with the left hand side of \eqref{ineq_iter} we get:
\begin{equation}\label{bound_x_k}
|(T^{u}_i(x))_{(j)}|^2 \ge \frac{2\lambda_{i}}{M_{i}} \qquad \forall
j \in \text{supp}(T^{u}_i(x)).
\end{equation}
Replacing $x = x^k$ for $k \ge 0$, it can be easily seen that, for
any $j \in \text{supp}(x^k_i)$ and $i \in [N]$, we have:
\begin{equation*}
\abs{x^k_{(j)}}^2
\begin{cases}
\ge \frac{2\lambda_{i}}{M_i}, &\text{if} \quad x^k_{(j)} \neq 0 \quad \text{and} \quad i \in \xi^k \\
= \abs{x^{0}_{(j)}}^2, &\text{if} \quad x^k_{(j)} \neq 0 \quad
\text{and} \quad i \notin \xi^k.
\end{cases}
\end{equation*}

\noindent Further, assume that at some iteration $k > 0$ a change of
sequence $I^k$ in expectation occurs. Thus, there is an index $j \in
[n]$ (and block $i$ containing $j$) such that either $\left(
x^k_{(j)}=0 \ \text{and} \ \left(T^{u}_i(x^{k})\right)_{(j)} \neq
0\right)$ or $\left( x^k_{(j)} \neq 0 \ \text{and} \
\left(T^{u}_i(x^{k})\right)_{(j)} = 0\right)$. Analyzing these cases
we have:
\begin{equation*}
\norm{T^{u}_i(x^k)-x^k_i}^2 \ge \left |
\left(T^{u}_i(x^k)\right)_{(j)}-x^k_{(j)} \right |^2 \;\;
\begin{cases}
 \ge \frac{2\lambda_{i}}{M_{i}} &\text{if} \quad x^k_{(j)} = 0 \\
\ge \frac{2\lambda_{i}}{M_{i}} &\text{if} \quad x^k_{(j)} \neq 0 \ \text{and} \ i \in \xi^k \\
= |x^{0}_{(j)}|^2 &\text{if} \quad x^k_{(j)} \neq 0 \ \text{and} \ i
\notin \xi^k.
\end{cases}
\end{equation*}

\noindent Observing that under uniform probabilities we have:
$$\mathbb{E}\left[\frac{\mu_{i_k}}{2}\norm{x^{k+1}-x^k}^2| x^k \right] =
\frac{1}{N}\sum\limits_{i=1}^N\frac{\mu_i}{2}\norm{T^{u}_i(x^k)-x^k_i}^2,$$
we can conclude that at each change of sequence $I^k$ in expectation we get:
\begin{align*}
\mathbb{E}\left[\frac{\mu_{i_k}}{2}\norm{x^{k+1}-x^k}^2 | x^k\right]  \ge \frac{1}{N}\min\left\{\min\limits_{i \in [N]: \lambda_i>0} \frac{\mu_i\lambda_i }{M_i},
 \min\limits_{i \in [N], j \in {\mathcal S}_i \cap \text{supp}(x^0)} \frac{ \mu_i}{2}
  |x^0_{(j)}|^2 \right\}.
\end{align*}

\noindent $(iii)$ From $\lim\limits_{k \to \infty}
\norm{x^{k+1}-x^k} = 0$ a.s. we have $\lim\limits_{k \to \infty}
\mathbb{E}\left[\norm{x^{k+1}-x^k} \ | \  x^k\right] = 0$ a.s. On
the other hand from part $(ii)$ we have that if the sequence $I^k$
changes in expectation,  then $\mathbb{E}[\norm{x^{k+1}-x^{k}}^2 \ |
\ x^{k}] \ge \delta > 0.$ These facts imply that there are a finite
number of changes in expectation of sequence $I^k$, i.e. there exist
$K>0$ such that for any $k > K$ we have $I^k = I^{k+1}$.

\noindent Further, if the sequence $I^k$ is constant for $k > K$,
then we have $I^k=I^*$ and $\norm{x^k}_{0,\lambda} =
\norm{x^*}_{0,\lambda}$ for any vector $x^*$ satisfying
$I(x^*)=I^*$. Also, for $k> K$ algorithm (RCD-IHT) is equivalent
with the classical random coordinate descent method
\cite{HonWan:13}, and thus shares its convergence properties, in
particular any limit point of the sequence $x^k$ is a minimizer  on
the coordinates $I^*$  for $\min_{x \in S_{I^*}} f(x)$. Therefore,
if the sequence $I^k$ is fixed, then we have for any $k > K$  and
$i_k \in I^k$:
\begin{equation}\label{fixedsupp}
u_{i_k}(x^{k+1}_{i_k};x^k)+ \norm{x^{k+1}}_{0,\lambda} \le
u_{i_k}(y_{i_k};x^k)+\norm{x^{k} +
U_{i_k}(y_{i_k}-x^k_{i_k})}_{0,\lambda} \quad \forall y_{i_k} \in
\rset^{n_{i_k}}.
\end{equation}

\noindent On the other hand, denoting with $x^*$ an accumulation
point of $x^k$, taking limit in \eqref{fixedsupp} and  using that
$\norm{x^k}_{0,\lambda} = \norm{x^*}_{0,\lambda}$ as $k \to \infty$,
we obtain the following relation:
$$F(x^*)\le \min_{y_{i} \in \rset^{n_i}}  u(y_i;x^*)+\norm{x^*+U_i(y_i-x^*_i)}_{0,\lambda} \quad a.s.$$
for all $i \in [N]$ and thus $x^*$ is the minimizer of  the previous
right hand side expression. Using the definition of local minimizers
from the set $\mathcal{L}_u$, we conclude that any limit point $x^*$
of the sequence $x^k$ belongs to this set, which proves our
statement.
\end{proof}

\vspace{3pt}

\noindent It is  important to note that the classical results for
any iterative algorithm used for solving nonconvex problems usually
state global convergence to stationary points, while for our
algorithms we were able to prove global convergence to local minima
of our nonconvex and NP-hard problem \eqref{l0regular}. Moreover, if
$\lambda_i=0$ for all $i \in [N]$, then the optimization problem
\eqref{l0regular} becomes convex and we see that our convergence
results cover also this setting.

%%%%%%%%%%%%%%%%%%%%%%%%%%%%%%%%%%%%%%%%%%%%%%%%%%%%%%%%%%%%%%%%%%%
%%%%%%%%%%%%%%%%%%%%%%%%%%%%%%%%%%%%%%%%%%%%%%%%%%%%%%%%%%%%%%%%%%%%

\section{Rate of convergence analysis}
\noindent In this section we prove the linear convergence  in
probability of the random coordinate descent algorithm (RCD-IHT)
under the additional assumption of  strong convexity for function
$f$ with parameter $\sigma$ and for the scalar case, i.e. we assume
$n_i=1$ for all $i \in [n] =[N]$. Note that, for algorithm (RCD-IHT)
the scalar case is the most practical since it requires solving a
simple unidimensional convex subproblem, while for $n_i > 1$ it
requires the solution of a small NP-hard subproblem at each
iteration. First, let us recall that complexity results of  random
block coordinate descent methods for solving convex problems $f^*
=\min_{x \in \rset^n} f(x)$,  under convexity and Lipschitz gradient
assumptions on the objective function, have been derived e.g. in
\cite{HonWan:13}, where the authors showed sublinear rate of
convergence  for a general class of coordinate descent methods.
Using a similar reasoning as in \cite{HonWan:13,NecPat:14a}, we
obtain that the randomized version of the general block coordinate
descent method, in the strongly convex case, presents a linear rate
of convergence in expectation of the~form:
\begin{equation*}
\mathbb{E}[f(x^{k}) - f^*] \le \left(1-\theta\right)^k
\left(f(x^{0}) - f^*\right),
\end{equation*}
where $\theta \in (0,1)$. Using the strong convexity property for $f$ we have:
\begin{equation}\label{rcd_rate_of_conv2}
\mathbb{E}\left[\norm{x^k - x^*} \right] \le \left(1- \theta\right)^{k/2}
 \sqrt{\frac{2}{\sigma}\left(f(x^{0}) - f^* \right)}  \quad \forall x \in
 X_f^*,
\end{equation}
where we recall that we denote $X_f^* = \arg \min_{x \in \rset^n}
f(x)$. For attaining an $\epsilon$-suboptimality this algorithm has
to perform the following number of iterations:
\begin{equation}\label{rcd_complexity2}
k \ge \frac{2}{\theta} \log
\frac{1}{\epsilon}\sqrt{\frac{2\left(f(x^0)-f^*\right)}{\sigma}}.
\end{equation}

\noindent In order to  derive the rate of convergence in probability
for algorithm (RCD-IHT), we first define the following notion which
is a generalization of relations \eqref{deltaq} and \eqref{deltae}
for $u_i(y_i,x) = u_i^q(y_i,x,M_i)$ and $u_i(y_i,x) =
u_i^e(y_i,x,\beta_i)$, respectively:
\begin{align}
\label{deltau1} & v^i(x) = x + U_i(h_i(x) - x_i), \quad \text{where}
\quad h_i(x) = \arg\min\limits_{y_i \in \rset^{}} u_i(y_i;x) \\
&\Delta^i(x) = u_i(0;x) - u_i(h_i(x);x). \label{deltau2}
\end{align}
We make the following assumption on functions $u_i$ and consequently
on  $\Delta^i(x)$:

\begin{assumption}
\label{assump_delta}
There exist some positive constants $C_i$ and $D_i$ such that the
approximation functions $u_i$ satisfy for all $i \in [n]$:
$$\abs{\Delta^i(x) - \Delta^i(z)} \le C_i\norm{x-z} + D_i\norm{x-z}^2  \quad \forall x \in
\rset^n, z \in \mathcal{T}_f$$
and
$$ \min_{z \in \mathcal{T}_f} \min\limits_{i \in [n]} \abs{\Delta^i(z) - \lambda_i} >0. $$
\end{assumption}

\noindent Note that if $f$ is strongly convex, then the set
$\mathcal{T}_f$ of basic local minima has a finite number of
elements. Next, we show that this assumption holds for the most
important approximation functions $u_i$ (recall that $u_i^q =u_i^Q$
in the scalar case $n_i=1$).
\begin{lemma}\label{delta_xy}
Under Assumption \ref{assump_grad_1}   the following statements
hold:\\
$(i)$ If we consider the separable quadratic approximation
$u_i(y_i;x) = u_i^q(y_i;x,M_i)$,~then:
$$ \abs{\Delta^i(x) - \Delta^i(z)} \le M_i v^i_{\max}\left(1 +
 \frac{L_f}{M_i}\right)\norm{x - z} + \frac{M_i}{2}\left(1 + \frac{L_f}{M_i}\right)^2 \norm{x - z}^2,$$
for all $x \in \rset^n$  and $z \in \mathcal{T}_f$, where we have
defined  $v^i_{\max}$ as follows $v^i_{\max} = \max
\{\norm{(v^i(y))_i} :\; y \in \mathcal{T}_f\}$ for all $i \in [n]$. \\
$(ii)$ If we consider the exact approximation $u_i(y_i;x) =
u_i^e(y_i;x,\beta_i)$, then we have:
\[  \abs{\Delta^i(x) - \Delta^i(z)} \le \gamma^i\norm{x-z} +
\frac{L_f+\beta_i}{2} \norm{x-z}^2,\] for all $x \in \rset^n$  and
$z \in \mathcal{T}_f$, where we have defined $\gamma^i$ as follows
$\gamma^i = \max\{\norm{\nabla f(y-U_iy_i)}+ \norm{\nabla f(v^i(y))}
+ \beta_i\norm{y_i}: \; y \in \mathcal{T}_f \}$  for all $i \in
[n]$.
\end{lemma}

\begin{proof}
$(i)$ For the separable quadratic approximation $u_i(y_i;x) =
u_i^q(y_i;x,M_i)$, using the definition of $\Delta^i(x)$ and
$v^i(x)$ given in \eqref{deltau1}--\eqref{deltau2} (see also
\eqref{deltaq}), we get:
\begin{align}
\label{deltaqq} \Delta^i(x) =
\frac{M_i}{2}\norm{x_i-\frac{1}{M_i}\nabla_if(x)}^2 =
\frac{M_i}{2}\norm{(v^i(x))_i}^2.
\end{align}

\noindent Then, since $\norm{\nabla_i f(x) - \nabla_i f(z)} \leq L_f
\norm{x-z}$ and using the property of the norm $|\norm{a} -
\norm{b}| \leq \norm{a-b}$ for any two vectors $a$ and $b$, we
obtain:
\begin{align*}
\abs{\Delta^i(x) -\Delta^i(z)} &= \frac{M_i}{2}\left\lvert \norm{(v^i(x))_i}^2 - \norm{(v^i(z))_i}^2\right\lvert \\
& \le \frac{M_i}{2}\left\lvert \norm{(v^i(x))_i} - \norm{(v^i(z))_i}\right\lvert \; \left\lvert \norm{(v^i(x))_i} + \norm{(v^i(z))_i}\right\lvert \\
& \overset{\eqref{deltaqq}}{\le} \frac{M_i}{2}\left(1 +
\frac{L_f}{M_i}\right)\norm{x - z} \left( 2\norm{(v^i(z))_i} +
\left(1 + \frac{L_f}{M_i}\right)\norm{x - z}\right).
\end{align*}

\noindent $(ii)$ For the exact approximation $u_i(y_i;x) =
u_i^e(y_i;x,\beta_i)$, using the definition of $\Delta^i(x)$ and
$v^i(x)$ given in \eqref{deltau1}--\eqref{deltau2} (see also
\eqref{deltae}),  we get:
\[ \Delta^i(x) = f(x-U_ix_i) - f(v^i(x)) +
\frac{\beta_i}{2}\norm{x_i}^2 -  \frac{\beta_i}{2}\norm{(v^i(x))_i -
x_i}^2. \] Then,  using the triangle inequality we derive the
following relation:
\begin{align*}
\abs{\Delta^i(x) - \Delta^i(z)} &\le \Big |f(x-U_ix_i)- f(z-U_iz_i) + f(v^i(z)) - f(v^i(x))\\
&\;\;\;\; + \frac{\beta_i}{2}\norm{(v^i(z))_i - z_i}^2 -
\frac{\beta_i}{2}\norm{(v^i(x))_i - x_i}^2 \Big |
+\Big\lvert\frac{\beta_i}{2}\norm{x_i}^2 -
\frac{\beta_i}{2}\norm{z_i}^2\Big\rvert.
\end{align*}

\noindent For simplicity, we denote:
\begin{align*}
\delta_{1i}(x,z)= &f(x-U_ix_i)- f(z-U_iz_i) + f(v^i(z)) - f(v^i(x)) \\
 & \;\; + \frac{\beta_i}{2}\norm{(v^i(z))_i - z_i}^2 - \frac{\beta_i}{2}\norm{(v^i(x))_i - x_i}^2\\
\delta_{2i}(x,z)=&\frac{\beta_i}{2}\norm{x_i}^2 -
\frac{\beta_i}{2}\norm{z_i}^2.
\end{align*}
\noindent In order to bound $\Delta^i(x) - \Delta^i(z)$, it is
sufficient to find  upper bounds on $\abs{\delta_{1i}(x,z)}$ and
$\abs{\delta_{2i}(x,z)}$. For a bound on $\abs{\delta_{1i}(x,z)}$ we
use $\abs{\delta_{1i}(x,y)} =
\max\{\delta_{1i}(x,y),-\delta_{1i}(x,y)\}$. Using the optimality
conditions for the map  $v^i(x)$ and convexity of $f$ we obtain:
\begin{align*}
 f(v^i(x))  & \ge f(v^i(z)) + \langle \nabla f(v^i(z)), v^i(x) - v^i(z)\rangle \nonumber\\
&= \! f(v^i(z)) \!+\! \langle \nabla f(v^i(z)), x \!-\! z\rangle
\!+\! \langle \nabla_i f(v^i(z)), ((v^i(x))_i \!-\! x_i)
 \!-\! ((v^i(z))_i \!-\! z_i)\rangle \nonumber\\
&=\! f(v^i(z)\!) \!+\! \langle \nabla f(v^i(z)\!), x \!-\! z\rangle
\!-\! \beta_i \langle (v^i(z)\!)_i \!-\! z_i,((v^i(x)\!)_i \!-\!
x_i)\!-\! ((v^i(z)\!)_i \!-\! z_i) \rangle \nonumber\\
&= f(v^i(z)) + \langle \nabla f(v^i(z)), x-z\rangle + \frac{\beta_i}{2}\norm{(v^i(z))_i - z_i}^2  \nonumber\\
 &\qquad  + \frac{\beta_i}{2}\norm{(v^i(z))_i-z_i}^2 - \beta_i \langle (v^i(z))_i -z_i, (v^i(x))_i-x_i\rangle \nonumber\\
&=f(v^i(z)) + \langle \nabla f(v^i(z)), x-z\rangle + \frac{\beta_i}{2}\norm{(v^i(z))_i-z_i}^2 \nonumber \\
 &\qquad + \frac{\beta_i}{2}\norm{(v^i(z))_i-z_i - ((v^i(x))_i-x_i)}^2 - \frac{\beta_i}{2}\norm{(v^i(x))_i - x_i}^2 \nonumber\\
& \ge \! f(v^i(z)) \!+\! \frac{\beta_i}{2}\norm{(v^i(z))_i \!-\!
z_i}^2 \!-\!  \frac{\beta_i}{2}\norm{(v^i(x))_i \!-\!  x_i}^2 \!-\!
\norm{\nabla f(v^i(z))}\norm{x \!-\! z}, %\label{bound_aux1}
\end{align*}
where in the last inequality we  used the Cauchy-Schwartz
inequality. On the other hand, from the global Lipschitz continuous
gradient inequality we get:
\begin{equation*}
f(x-U_ix_i)  \le f(z-U_iz_i) + \norm{\nabla f(z-U_iz_i)}\norm{x-z} +
\frac{L_f}{2}\norm{x-z}^2.
\end{equation*}
\noindent From previous two relations  we obtain:
\begin{equation}\label{delta}
\delta_{1i}(x,z) \le \left( \norm{\nabla f(z-U_iz_i)}+ \norm{\nabla
f(v^i(z))} \right) \norm{x-z} + \frac{L_f}{2}\norm{x-z}^2.
\end{equation}

\noindent In order to obtain a bound on $-\delta_{1i}(x,z)$ we
observe that:
\begin{align}
& f(v^i(x)) + \frac{\beta_i}{2}\norm{(v^i(x))_i-x_i}^2 - f(v^i(z)) - \frac{\beta_i}{2}\norm{(v^i(z))_i
-z_i}^2 \nonumber\\
&\quad  \le f(x + U_i((v^i(z))_i-z_i)) - f(v^i(z)) \nonumber \\
& \quad \le \norm{\nabla f(v^i(z))}\norm{x-z} +
\frac{L_f}{2}\norm{x-z}^2,\label{bound_aux3}
\end{align}
where in the last inequality we  used the Lipschitz gradient
relation and Cauchy-Schwartz inequality. Also, from the convexity of
$f$ and the Cauchy-Schwartz inequality we get:
\begin{equation}\label{bound_aux4}
f(x-U_ix_i) \ge f(z-U_iz_i) - \norm{\nabla f(z-U_iz_i)}\norm{x-z}.
\end{equation}
Combining now the bounds \eqref{bound_aux3} and \eqref{bound_aux4}
we obtain:
\begin{equation}\label{-delta}
 -\delta_{1i}(x,z)\le \left(\norm{\nabla f(z-U_iz_i)}+ \norm{\nabla f(v^i(z))}\right)\norm{x-z} + \frac{L_f}{2}\norm{x-z}^2.
\end{equation}
Therefore, from \eqref{delta} and \eqref{-delta} we obtain a bound
on $\delta_{1i}(x,z)$:
\begin{equation}\label{delta1abs}
\abs{\delta_{1i} (x,z)}\le \left( \norm{\nabla f(z-U_iz_i)}+
\norm{\nabla f(v^i(z))} \right) \norm{x-z} +
\frac{L_f}{2}\norm{x-z}^2.
\end{equation}

\noindent Regarding the second quantity $\delta_{2i}(x,z)$, we
observe that:
\begin{align}
 \abs{\delta_{2i}(x,z)} &= \frac{\beta_i}{2} \Big\lvert\norm{x_i}+\norm{z_i}
 \Big\rvert \Big\lvert \norm{x_i}-\norm{z_i}\Big\rvert
  = \frac{\beta_i}{2}\Big\lvert \norm{x_i}-\norm{z_i}+2\norm{z_i}\Big\rvert
   \Big\lvert\norm{x_i}-\norm{z_i}\Big\rvert \nonumber\\
% & \le \frac{\beta_i}{2}\left(\norm{x_i - z_i}+2\norm{z_i}\right) \norm{x_i - z_i}\nonumber\\
 & \le \frac{\beta_i}{2}\left(\norm{x - z}+2\norm{z_i}\right) \norm{x - z}.\label{delta2abs}
\end{align}
From the upper bounds on $\abs{\delta_{1i} (x,z)}$ and
$\abs{\delta_{2i} (x,z)}$ given in  \eqref{delta1abs} and
\eqref{delta2abs}, respectively,  we obtained our result.
\end{proof}

\noindent We further show  that  the second part of Assumption \ref{assump_delta} holds for the most important  approximation functions $u_i$.
\begin{lemma}\label{lemma_alpha}
Under Assumption \ref{assump_grad_1}  the following statements
hold:\\
$(i)$ Considering the separable quadratic approximation $u_i(y_i;x)
= u_i^q(y_i;x,M_i)$, then for any fixed $z \in \mathcal{T}_f$
there exist only two values of parameter $M_i$ satisfying $\abs{\Delta^i(z) - \lambda_i} = 0$.\\
$(ii)$ Considering the exact approximation $u_i(y_i;x) =
u_i^e(y_i;x,\beta_i)$, then for any fixed $z \in \mathcal{T}_f$,
there exists a unique $\beta_i$ satisfying $\abs{\Delta^i(z) - \lambda_i} = 0$.\\
\end{lemma}
\begin{proof}
$(i)$ For  the approximation $u_i(y_i;x) = u_i^q(y_i;x,M_i)$  we  have: $$\Delta^i(z) = \frac{M_i}{2}\norm{z_i - \frac{1}{M_i}\nabla_i f(z)}^2.$$
Thus, we observe that $\Delta^i(z) = \lambda_i$ is equivalent with the following relation:
$$ \frac{\norm{z_i}^2 }{2}M_i^2 - \left(  \langle \nabla_i f(z), z_i\rangle +\lambda_i \right)M_i + \frac{\norm{\nabla_i f(z)}^2}{2} = 0.$$
which is valid for only two values of $M_i$.

\noindent $(ii)$ For the approximation $u_i(y_i;x) = u_i^e(y_i;x,\beta_i)$ we have:
$$\Delta^i(z) = f(z-U_iz_i) + \frac{\beta_i}{2}\norm{z_i}^2 - f(v^i_{\beta}(z)) -
\frac{\beta_i}{2}\norm{h^i_{\beta}(z) - z_i}^2,$$ where
$v^i_{\beta}(z)$  and  $h^i_{\beta}(z)$ are defined as in
\eqref{deltau1} corresponding to the exact approximation. Without
loss of generality, we can assume that there exist two constants
$\beta_i > \gamma_i > 0$ such that $\Delta^i(z) = \lambda_i$. In
other terms, we have:
$$ \frac{\beta_i}{2}\norm{z_i}^2 - f(v^i_{\beta}(z)) - \frac{\beta_i}{2}\norm{h^i_{\beta}(z) - z_i}^2 =
\frac{\gamma_i}{2}\norm{z_i}^2 - f(v^i_{\gamma}(z)) -
\frac{\gamma_i}{2}\norm{h^i_{\gamma}(z) - z_i}^2.$$ We analyze two
possible cases. Firstly, if $z_i = 0$, then the above equality leads
to the following relation:
\begin{align*}
f(v^i_{\beta}(z)) + \frac{\beta_i}{2}\norm{h^i_{\beta}(z)}^2 &=
f(v^i_{\gamma}(z)) + \frac{\gamma_i}{2}\norm{h^i_{\gamma}(z)}^2 \\
& \le  f(v^i_{\beta}(z)) +
\frac{\gamma_i}{2}\norm{h^i_{\beta}(z)}^2, \end{align*} which
implies that $\beta_i \le \gamma_i$, that is a contradiction.
Secondly, assuming $z_i \neq 0$ we observe from optimality of
$h^i_{\beta}(z)$ that:
\begin{equation}\label{ineq1}
 \frac{\beta_i}{2}\norm{z_i}^2 - f(v^i_{\beta}(z)) - \frac{\beta_i}{2}\norm{h^i_{\beta}(z) - z_i}^2 \ge  \frac{\beta_i}{2}\norm{z_i}^2 - f(z).
\end{equation}
On the other hand, taking into account that $z \in \mathcal{T}_f$ we
have:
\begin{equation}\label{ineq2}
\frac{\gamma_i}{2}\norm{z_i}^2 - f(v^i_{\gamma}(z)) -
\frac{\gamma_i}{2}\norm{h^i_{\gamma}(z) - z_i}^2 \le
\frac{\gamma_i}{2}\norm{z_i}^2 - f(z).
\end{equation}
From \eqref{ineq1} and \eqref{ineq2} we get $\beta_i \le \gamma_i$, thus
implying the same contradiction.
\end{proof}

\noindent We use the following notations:
\begin{align*}
C_{\max} = \max_{1\le i \le n} C_i,\quad D_{\max} = \max_{1\le i \le
n} D_i,\quad \tilde{\alpha} &= \min_{z \in \mathcal{T}_f}
\min\limits_{i \in [n]} \abs{\Delta^i(z) - \lambda_i}.
\end{align*}
\noindent Since the cardinality of basic local minima
$\mathcal{T}_f$ is finite for strongly convex functions $f$, then
there is a finite number of possible values for $\abs{\Delta^i(z) -
\lambda_i}$. Therefore, from previous lemma we obtain that
$\tilde{\alpha}=0$ for a finite number of values of parameters
$(M_i,\mu_i)$ of the approximations $u_i = u_i^q$ or $u_i=u_i^e$. We
can reason in a similar fashion for general approximations $u_i$,
i.e. that $\tilde{\alpha}=0$ for a finite number of values of
parameters $(M_i,\mu_i)$ of the approximations $u_i$ satisfying
Assumption \ref{approximation}. In conclusion, choosing randomly at
an initialization stage of our algorithm the parameters
$(M_i,\mu_i)$ of the approximations $u_i$, we can conclude that
$\tilde{\alpha}>0$ almost sure.

\noindent Further, we state the linear rate of convergence with high
probability for algorithm (RCD-IHT). Our analysis will employ ideas
from the convergence proof of deterministic iterative hard
thresholding method in \cite{Lu:12}. However, the random nature of
our family of methods and the  properties of the approximation
functions $u_i$ require a new approach. We use the notation $k_p$
for the iterations when a change in expectation of $I^k$ occurs, as
given in the previous section. We also denote with $F^*$ the global
optimal value of our original $\ell_0$ regularized problem
\eqref{l0regular}.

%%%%%%%%%%%%%%%%%%%%%%%%%%%%%%%%%%%%%%%%%%%%%%%%%%%%%%%%%%%%%%%%%%%%%%%%%%%%%%%%%%%%%%%%%%%%%%%%%%%%%%%%%%%%%%%%%%%%%%%%%%%

\begin{theorem}
Let $x^k$ be the sequence generated by the family of algorithms
(RCD-IHT) under Assumptions  \ref{assump_grad_1},
\ref{approximation} and \ref{assump_delta} and the additional assumption of strong convexity of
$f$ with parameter $\sigma$. Denote with $\kappa$ the number of
changes in expectation of $I^k$ as $k \to \infty$. Let $x^*$ be some
limit point of $x^k$ and $\rho>0$ be some confidence level.
Considering the scalar case $n_i=1$ for all $i \in [n]$, the
following statements hold:

\noindent \textit{(i)} The number of changes in expectation $\kappa$
of $I^k$ is bounded by $\left \lceil \frac{ \mathbb{E} \left [
F(x^0) - F(x^*) \right] }{\delta} \right \rceil$,  where $\delta$ is
specified in Theorem \ref{convergence_rpamiht} $(ii)$.

%\noindent  \textit{(ii)} The sequence $x^k$ satisfies
%$\mathbb{P}\left(F(x^k) - F(x^*) \le \epsilon \right) \ge 1-\rho$
%for any iteration counter  $k \ge
%\frac{2}{\theta}\log\frac{\tilde{\omega}}{\rho\epsilon}$, where
%$\tilde{\omega}=2^{\frac{\theta\omega}{2} + \log \delta}$ and
%$\theta$ given in \eqref{rcd_rate_of_conv2}, with $\omega \!\!=\!\!
%\left\{\max\limits_{t \in \rset^{}} \ \alpha t -\beta t^2 : \;  0
%\le t \le \left [ \frac{\mathbb{E}[F(x^0) - F(x^*)]}{\delta} \right
%]  \right\}, \beta \!=\! \frac{\delta}{4(F(x^0)-F^*)}$, $\alpha =
%\frac{1}{\theta}\left(\log \left[2(F(x^0) - F^*)\right]  +
%2\log\frac{2 n }{\sqrt{\sigma}\xi} + \frac{\delta}{4
%(F(x^0)-F^*)}\right)$ and
%$\xi=\frac{1}{2}\left(\sqrt{\frac{C_{\max}^2}{D_{\max}^2} +
%\frac{\tilde{\alpha}}{D_{\max}}} -
%\frac{C_{\max}}{D_{\max}}\right)$.

\noindent \textit{(ii)} The sequence $x^k$ converges linearly in the
objective function values with high probability, i.e. it satisfies
$\mathbb{P}\left(F(x^k) - F(x^*) \le \epsilon \right) \ge 1-\rho$
for $k \ge \frac{1}{\theta}\log\frac{\tilde{\omega}}{\rho\epsilon}$,
where $\tilde{\omega}=2^{\omega}(F(x^0)-F^*)$, with $\omega =
\left\{\max\limits_{t \in \rset^{}} \ \alpha t -\beta t^2 : 0 \le t
\le \left\lfloor\frac{\mathbb{E}[F(x^0) -
F(x^*)]}{\delta}\right\rfloor \right\}, \beta =
\frac{\delta}{2(F(x^0)-F^*)}$, $\alpha = \left(\log \left[2(F(x^0) -
F^*)\right]  + 2\log\frac{2 N }{\sqrt{\sigma}\xi}  - \frac{\delta}{2
(F(x^0)-F^*)} + \theta \right)$ and
$\xi=\frac{1}{2}\left(\sqrt{\frac{C_{\max}^2}{D_{\max}^2} +
\frac{\tilde{\alpha}}{D_{\max}}} -
\frac{C_{\max}}{D_{\max}}\right)$.
\end{theorem}

\begin{proof}
$(i)$ From \eqref{decrease} and Theorem \ref{convergence_rpamiht}
$(ii)$ it can be easily seen that:
\begin{align*}
\delta & \le
\mathbb{E}\left[\frac{\mu_{i_{k_p}}}{2}\norm{x^{k_p+1}-x^{k_p}}^2
\Big| x^{k_p}\right] \le F(x^{k_p}) -
\mathbb{E}[F(x^{k_p+1})|x^{k_p}] \\
& \leq  F(x^{k_p}) - \mathbb{E}[F(x^{k_{p+1}})|x^{k_p}].
\end{align*}
Taking expectation in this relation w.r.t. the entire history
$\xi^{k_p}$ we get the bound: $\delta \le \mathbb{E}\left[
F(x^{k_p}) - F(x^{k_{p+1}}) \right].$ Further, summing up over $p
\in [\kappa]$  we have:
$$\kappa \delta  \le \mathbb{E}\left[F(x^{k_1}) -
F(x^{k_{\kappa}+1})\right]\le \mathbb{E}\left[F(x^0) -
F(x^*)\right],$$ i.e. we have proved the first part of our theorem.

\noindent $(ii)$ In order to establish the linear rate of
convergence in probability of algorithm (RCD-IHT), we first derive a
bound on the number of iterations performed between two changes in
expectation of $I^k$. Secondly, we also derive a bound on  the
number of iterations performed after the support is fixed (a similar
analysis  for deterministic iterative hard thresholding method was
given  in \cite{Lu:12}). Combining these two bounds, we obtain the
linear convergence of our algorithm. Recall that for any $p \in
[\kappa]$, at iteration $k_{p}+1$, there is a change in expectation
of $I^{k_p}$, i.e.
\begin{equation*}
\mathbb{E}[\abs{I^{k_{p}} \setminus I^{k_{p}+1}} + \abs{I^{k_{p}+1}
\setminus I^{k_{p}}} \Big| \ x^{k_{p}}] > 0,
\end{equation*}
which implies that
$$\mathbb{P}\left( \abs{I^{k_{p}} \setminus I^{k_{p}+1}} + \abs{I^{k_{p}+1} \setminus
 I^{k_{p}}} > 0 | x^{k_p}\right) = \mathbb{P}\left( I^{k_{p}} \neq I^{k_{p}+1} | x^{k_p}\right)
  \ge \frac{1}{n}$$
and furthermore
\begin{equation}\label{probabilitysupp2}
 \mathbb{P}\left( \abs{I^{k_{p}} \setminus I^{k_{p}+1}} + \abs{I^{k_{p}+1}
 \setminus I^{k_{p}}} = 0 | x^{k_p} \right)  = \mathbb{P}\left( I^{k_{p}} =
 I^{k_{p}+1} | x^{k_p}\right) \le \frac{n-1}{n}.
\end{equation}
\noindent Let $p$ be an arbitrary integer from $[\kappa]$. Denote
$\hat{x}^* = \arg\min\limits_{x \in S_{I^{k_p}}} f(x)$ and
$\hat{f}^*=\mathbb{E}\left[f(\hat{x}^*) \ | \ x^{k_{p-1}+1} \right]$.

\noindent Assume that the number of iterations performed between two
changes in expectation satisfies:
\begin{equation}\label{iter_rcd_aux2}
k_{p} - k_{p-1} > \frac{1}{\theta} \left(\log \left[2(F(x^0) - F^* -
(p-1)\delta )\right] + 2\log\frac{2 n }{\sqrt{\sigma}\xi} \right) +
1,
\end{equation}
where we recall that $\sigma$ is the strong convexity parameter of
$f$. For any  $k \in [k_{p-1}+1, k_{p}]$ we denote $f^k =\mathbb{E}[f(x^k) \ | \ x^{k_{p-1}+1}]$. From Lemma \ref{descent_ramiht} and Theorem
\ref{convergence_rpamiht} we have:
$$f^{k_{p-1}+1} - \hat{f}^* \le \mathbb{E}[F(x^{k_{p-1}+1}) \ | \ x^{k_{p-1}+1}] - \mathbb{E}[F(\hat{x}^*)\ | \ x^{k_{p-1}+1}]
 \le F(x^0) - (p-1)\delta - F^*,$$

%\textcolor{red}{\begin{small} [Remark: The questioned inequality
%should result from Lemma \ref{descent_ramiht} if the expectation is
%unconditioned. However, if we take conditional expectation w.r.t.
%$x^{k_{p-1}+1}$, then is not clear if the inequality
%holds.]\end{small}}

\noindent so that we can  claim that \eqref{iter_rcd_aux2} implies
\begin{equation}\label{iter_rcd2}
k_{p} - k_{p-1} > \frac{2}{\theta} \log \frac{
2\sqrt{2(f^{k_{p-1}+1} - \hat{f}^*)}n}{\sqrt{\sigma}\xi} + 1 \ge
\frac{2}{\theta} \log \frac{ \sqrt{2n(f^{k_{p-1}+1} -
\hat{f}^*)}}{\sqrt{\sigma}\xi(\sqrt{n}-\sqrt{n-1})} +1.
\end{equation}
\noindent We show that under relation \eqref{iter_rcd2}, the
probability \eqref{probabilitysupp2} does not hold. First, we
observe that between two changes in expectation of $I^k$, i.e.
$k \in [k_{p-1}+1, k_{p}]$, the  algorithm (RCD-IHT) is equivalent with
the randomized version of coordinate descent method \cite{HonWan:13,
NecPat:14a} for strongly convex problems. Therefore, the method has
linear rate of convergence \eqref{rcd_rate_of_conv2}, which in our
case is given by the following expression:
\begin{equation*}
\mathbb{E}\left[\norm{x^k\!-\!\hat{x}^*} \;|\; x^{k_{p-1}+1} \right]\!\le\!\left(1\!-\!\theta\right)^{(k-k_{p-1}-1)/2}\sqrt{\frac{2}{\sigma}\left(f^{k_{p-1}+1}-\hat{f}^*\right)},
\end{equation*}
for all $k \in [k_{p-1}+1, k_{p}]$. Taking $k=k_{p}$, if we apply
the complexity estimate \eqref{rcd_complexity2} and use the bound
\eqref{iter_rcd2}, we obtain:
$$\mathbb{E}\left[\norm{x^{k_{p}} - \hat{x}^*} \;|\; x^{k_{p-1}+1} \right] \le \left(1 - \theta\right)^{(k_{p} - k_{p-1}-1)/2}
\!\sqrt{\frac{2}{\sigma}\left(f^{k_{p-1}+1}-\hat{f}^*\right)} \!<\!
\!\xi\!\left(1\!-\!\sqrt{\frac{n\!-\!1}{n}}\right).$$ From the Markov
inequality, it can be easily seen that we have:
\begin{equation*}
\mathbb{P}\left(\norm{x^{k_{p}} - \hat{x}^*} < \xi \;|\; x^{k_{p-1}+1} \right) = 1-
\mathbb{P}\left(\norm{x^{k_{p}} - \hat{x}^*} \ge \xi \;|\; x^{k_{p-1}+1}\right) >
\sqrt{1-\frac{1}{n}}.
\end{equation*}

\noindent Let $i \in [N]$ such that $\lambda_i>0$. From  Assumption
\ref{assump_delta} and definition of parameter $\xi$ we see that the
event $\norm{x^{k_{p}} - \hat{x}^*} < \xi$ implies:
\begin{equation*}
\abs{\Delta^i(x^{k_p}) - \Delta^i(\hat{x}^*)} \le
C_{\max}\norm{x^{k_p}-\hat{x}^*} +
D_{\max}\norm{x^{k_p}-\hat{x}^*}^2 < \tilde{\alpha} \le
\abs{\Delta^i(\hat{x}^*) - \lambda_i}.
\end{equation*}

\noindent The first and the last terms from the above inequality
further imply:
\begin{equation*}
\begin{cases}
 \abs{\Delta^i(x^{k_p})} > \lambda_i, & \text{if} \quad  \abs{\Delta^i(\hat{x}^*)} > \lambda_i\\
 \abs{\Delta^i(x^{k_p})} < \lambda_i, & \text{if} \quad  \abs{\Delta^i(\hat{x}^*)} < \lambda_i,
\end{cases}
\end{equation*}
or equivalently  $I^{k_p+1} = \hat{I}^* = \left\{ j \in [n]:
\lambda_j=0 \right\} \cup \left\{ i \in [n]: \lambda_i>0,
\abs{\Delta^i(\hat{x}^*)} > \lambda_i \right\}$. \noindent In
conclusion, if \eqref{iter_rcd2} holds, then we have:
\begin{equation*}
\mathbb{P}\left( I^{k_{p}+1} = \hat{I}^* \;|\; x^{k_{p-1}+1} \right) >
\sqrt{1-\frac{1}{n}}.
\end{equation*}
Applying the same procedure as before for iteration $k = k_{p} - 1$
we obtain:
\begin{equation*}
\mathbb{P}\left( I^{k_{p}} = \hat{I}^* \;|\; x^{k_{p-1}+1}\right) >
\sqrt{1-\frac{1}{n}}.
\end{equation*}

%\textcolor{red}{Considering the events $\{I^{k_{p}} = \hat{I}^*\}$
%and $\{I^{k_{p}+1} = \hat{I}^*\}$ to be independent
%\begin{small}[Remark: this assumption requires stronger
%arguments!]\end{small} }, we have:

\noindent Considering the events $\{I^{k_{p}} = \hat{I}^*\}$  and
$\{I^{k_{p}+1} = \hat{I}^*\}$ to be independent (according to the
definition of $k_p$), we have:
\begin{equation*}
\mathbb{P}\left( \left\{I^{k_{p}+1} = \hat{I}^*\right\} \cap
\left\{I^{k_{p}} = \hat{I}^*\right\} \;|\; x^{k_{p-1}+1} \right) = \mathbb{P}\left(
I^{k_{p}+1} = I^{k_{p}} \;|\; x^{k_{p-1}+1} \right) > \frac{n-1}{n},
\end{equation*}
which  contradicts the assumption $\mathbb{P}\left(I^{k_{p}} =
I^{k_{p}+1} \ | \ x^{k_{p}}\right) \le \frac{n-1}{n}$ (see
\eqref{probabilitysupp2} and the definition of $k_p$ regarding the
support of $x$).

%\textcolor{red}{\begin{small}[Remark: note that the contradiction is
%not clear, because the first inequality involves a conditional
%probability w.r.t. $x^{k_{p-1}+1}$ and the second inequality
%involves a probability conditioned by $x^{k_p}$!]\end{small}}

\noindent Therefore, between two changes of support the number of
iterations is bounded by:
\begin{equation*}
k_{p} - k_{p-1} \le \frac{1}{\theta} \left(\log \left[2(F(x^0) - F^*
- (p-1)\delta )\right] + 2\log\frac{2 n }{\sqrt{\sigma}\xi} \right)
+1.
\end{equation*}
We can further derive the following:
\begin{align*}
&\frac{1}{\theta} \left(\log \left[2(F(x^0) - F^* - (p-1)\delta )\right] +
 2\log\frac{2 n }{\sqrt{\sigma}\xi} \right) \\
&= \frac{1}{\theta} \left(\log \left[2 (F(x^0) - F^*)\left(1 -
\frac{(p-1)\delta}{F(x^0)-F^*}
\right)\right] + 2\log\frac{2 n }{\sqrt{\sigma}\xi} \right) \\
& = \frac{1}{\theta} \left(\log \left[2 (F(x^0) - F^*)\right] +
 \log\left[1 - \frac{(p-1)\delta}{F(x^0)-F^*}\right] + 2\log\frac{2 n }{\sqrt{\sigma}\xi}
 \right)  \\
&\le \frac{1}{\theta} \left(\log \left[2(F(x^0) - F^*)\right]  -
\frac{(p-1)\delta}{F(x^0)-F^*} + 2\log\frac{2 n }{\sqrt{\sigma}\xi}
\right),
\end{align*}
\noindent where we  used the inequality $\log(1-t) \le -t$ for any
$t \in (0, \ 1)$. Denoting with $k_\kappa$ the number of iterations
until the last change of support, we have:
\begin{align*}
k_\kappa & \le \sum\limits_{p=1}^{\kappa}\frac{1}{\theta} \left(\log
\left[2
 (F(x^0) - F^*)\right]  - \frac{(p-1)\delta}{F(x^0)-F^*}
  + 2\log\frac{2 n }{\sqrt{\sigma}\xi}  \right) +1 \\
& = \kappa \frac{1}{\theta}\left(\log \left[2(F(x^0) - F^*)\right]
 + 2\log\frac{2 n }{\sqrt{\sigma}\xi}  + \frac{\delta}{2 (F(x^0)-F^*)} + \theta \right)
 - \frac{\kappa^2}{\theta}
 \underbrace{\frac{\delta}{2(F(x^0)-F^*)}}_{\beta}.
\end{align*}

\noindent Once the support is fixed (i.e. after $k_\kappa$
iterations), in order to reach some $\epsilon$-local minimum in
probability with some confidence level $\rho$, the algorithm
(RCD-IHT) has to perform additionally another
$$\frac{1}{\theta}\log \frac{f^{k_\kappa+1}-f(x^*)}{\epsilon \rho}$$
iterations, where we used again \eqref{rcd_complexity2} and Markov
inequality. Taking into account that the iteration $k_\kappa$ is the
largest possible integer at which the support of sequence $x^k$
could change, we can bound:
$$f^{k_\kappa + 1}- f(x^*) = E[F(x^{k_\kappa + 1})-
F(x^*)] \le F(x^0) - F^* - \kappa\delta.$$

\noindent Thus, we obtain:
\begin{align*}
\frac{1}{\theta}&\log \frac{f^{k_\kappa+1}-f(x^*)}{\epsilon \rho} \le \frac{1}{\theta}\log \frac{F(x^0) - F^* - \kappa\delta}{\epsilon \rho} \\
&\le \frac{1}{\theta}\left(\log \left[(F(x^0) - F^*)\left(1 -  \frac{\kappa\delta}{F(x^0) - F^*}\right)\right]  - \log \epsilon \rho \right)\\
&\overset{\log (1-t) \leq -t}{\le} \frac{1}{\theta}\left(\log (F(x^0) - F^*)  -  \frac{\kappa\delta}{F(x^0) - F^*}  - \log \epsilon \rho \right)\\
&\le \frac{1}{\theta}\left(\log \frac{F(x^0) - F^*}{\epsilon \rho }
-  \frac{\kappa\delta}{F(x^0) - F^*} \right).
\end{align*}
Adding up this quantity and the upper bound on $k_{\kappa}$, we get
that the algorithm (RCD-IHT) has to perform at most
$$\frac{1}{\theta} \left(\alpha\kappa - \beta \kappa^2 + \log \frac{F(x^0) - F^*}{\epsilon \rho }\right)\le \frac{1}{\theta} \left(\omega + \log \frac{F(x^0) - F^*}{\epsilon \rho }\right)$$
iterations in order to attain an $\epsilon$-suboptimal point with
probability at least $\rho$, which proves the second statement of
our theorem.
\end{proof}

\noindent Note that we have obtained  global linear convergence for
our family of random coordinate descent methods on the  class of
$\ell_0$ regularized problems with strongly convex objective
function $f$.

%%%%%%%%%%%%%%%%%%%%%%%%%%%%%%%%%%%%%%%%%%%%%%%%%%%%%%%%%%%%%%%%%%%%%%
%%%%%%%%%%%%%%%%%%%%%%%%%%%%%%%%%%%%%%%%%%%%%%%%%%%%%%%%%%%%%%%%%%%%%%%%%%5

\section{Random data experiments on sparse learning} In this
section we analyze the practical performances of our family of
algorithms  (RCD-IHT) and compare them  with that of algorithm
(IHTA) \cite{Lu:12}. We perform several numerical tests on sparse
learning problems with randomly generated data. All  algorithms were
implemented in Matlab code and the numerical simulations are
performed on a PC with Intel Xeon E5410 CPU and 8 Gb RAM memory.

\noindent Sparse learning represents  a  collection of learning
methods which seek a tradeoff between some goodness-of-fit measure
and sparsity of the result, the latter property allowing better
interpretability. One of the models widely used in machine learning
and statistics is the linear model (least squares setting). Thus, in
the first set of tests   we consider sparse linear~formulation:

$$\min\limits_{x \in \rset^n} F(x) \quad \left(=\frac{1}{2}\norm{Ax-b}^2 + \lambda \norm{x}_0 \right),$$
where $A \in \rset^{m \times n} $ and $\lambda >0$.  We analyze  the
practical efficiency of our algorithms in terms of  the probability
of  reaching a global optimal point. Due to difficulty of finding
the global solution of this problem, we consider a small model $m=6$
and $n=12$. For each penalty parameter $\lambda $, ranging from
small values (0.01) to large values (2), we ran the family of
algorithms (RCD-IHT), for separable quadratic approximation (denoted
(RCD-IHT-$u^q$), for exact approximation (denoted (RCD-IHT-$u^e$)
and (IHTA) \cite{Lu:12} from 100 randomly generated (with random
support) initial vectors. The numbers of runs out of 100 in which
each method found the global optimum is given in Table \ref{tabel2}.
We observe that for all values of $\lambda$  our algorithms
(RC-IHT-$u^q$) and (RCD-IHT-$u^e$) are able to identify the global
optimum with a rate of success  superior to algorithm (IHTA) and for
extreme values of $\lambda$ our algorithms perform much better than
(IHTA).

\renewcommand{\tabcolsep}{4pt}
\begin{table}[ht]
\centering \caption{Numbers of runs out of 100 in which algorithms
(IHTA), (RCD-IHT-$u^q$) and (RCD-IHT-$u^e$) found  global optimum.}
{\small \label{tabel2}
\begin{tabular}{|c|c|c|c|}
\hline
$\lambda $  &\textbf{(IHTA)} & \textbf{(RCD-IHT-$u^q$)} & \textbf{(RCD-IHT-$u^e$)}\\
\hline \hline
$0.01$ & 95 &  96  & 100\\
\hline
$0.07$ & 92 &  92 & 100\\
\hline
$0.09$ & 43 &  51 & 70\\
\hline
$0.15$ & 41 &  47 & 66\\
\hline
$0.35$ & 24 &  28 & 31\\
\hline
$0.8$ & 36 &  43 & 44\\
\hline
$1.2$ & 29 &  29 & 54\\
\hline
$1.8$ & 76 &  81 & 91\\
\hline
$2$  & 79  &  86  & 97 \\
\hline
\end{tabular}
}
\end{table}

\vspace{0.2cm}

\noindent In the second set of experiments  we consider  the
$\ell_2$ regularized logistic loss model from machine learning
\cite{Bah:13}. In this model the relation between the data,
represented by a random vector $a \in \rset^n$,  and its associated
label, represented by a random binary variable $y \in \{0, 1\}$, is
determined by the conditional probability:
$$P\{y | a;x \}= \frac{e^{y \langle a,x\rangle}}{1+e^{\langle a,x\rangle}},$$
where $x$ denotes a parameter vector. Then, for a set of $m$
independently drawn  data samples $\{(a_i , y_i )\}_{i=1}^m$, the
joint likelihood can be written as a function of $x$. To find the
maximum likelihood estimate one should maximize the likelihood
function, or equivalently minimize the negative log-likelihood (the
logistic loss):
$$\min\limits_{x \in \rset^n} \frac{1}{m}\sum\limits_{i=1}^m
\log\left(1 + e^{\langle a_i,x\rangle} \right) - y_i\langle a_i,x
\rangle. $$ Under the assumption of $n \le m$ and $A = \left[a_1,
\dots, a_m \right] \in \rset^{n \times m}$  being full rank, it is
well known that $f(\cdot)$ is strictly convex. However, there are
important applications (e.g. feature selection)  where these
assumptions are not satisfied and the problem is highly ill-posed.
In order to compensate this drawback, the logistic loss is
regularized by some penalty term (e.g. $\ell_2$ norm $\norm{x}^2_2$,
see \cite{Bah:13,Has:09}). Furthermore, the penalty term implicitly
bounds the length of the minimizer, but does not promote sparse
solutions. Therefore, it is desirable to impose an additional
sparsity regularizer, such as the $\ell_0$ quasinorm. In conclusion
our problem to be minimized is given by:
$$\min\limits_{x \in \rset^n} F(x) \quad \left(=\frac{1}{m}\sum\limits_{i=1}^m \log\left(1 + e^{\langle a_i,x\rangle} \right)
- y_i\langle a_i,x \rangle + \frac{\nu}{2}\norm{x}^2 +
\norm{x}_{0,\lambda}\right),$$ where now $f$ is strongly convex with
parameter $\nu$. For simulation, data were uniformly random
generated and we fixed the  parameters $\nu=0.5$  and $\lambda=0.2$.
Once an instance of random data has been generated, we ran  10 times
our algorithms (RCC-IHT-$u^q$) and  (RCD-IHT-$u^e$) and  algorithm
(IHTA) \cite{Lu:12} starting from 10 different initial points. We
reported in Table \ref{tabel1} the best results of each algorithm
obtained over all 10 trials, in terms of best function value that
has been attained with associated sparsity and number of iterations.
%As a stopping criterion for (IHTA) we have used the common descent criterion $f(x^k) - %f(x^{k+1}) \le \epsilon$. However, we stopped both (RC-IHT) and (RCD-IHT) if there has %been achieved a certain number (chosen $2n$) of successive zero descent iterations (i.e. %$f(x^k)=f(x^{k+1})$).In general, due to simplicity of the iterations of coordinate descent, %the number of iterations performed by the coordinate descent algorithms
%is significantly larger than the number of iterations of usual first order methods.
In order to report relevant information, we have measured the
performance   of coordinate descent methods (RCD-IHT-$u^q$) and
(RCD-IHT-$u^e$) in terms of full iterations obtained by dividing the
number of all iterations by the dimension $n$. The column $F^*$
denotes the final function value attained by the algorithms,
$\norm{x^*}_0$ represents the sparsity of the last generated  point
and \textit{iter} (\textit{full-iter}) represents the number of
iterations (the number of full iterations). Note that our algorithms
(RCD-IHT-$u^q$) and (RCD-IHT-$u^e$) have superior performance in
comparison with algorithm (IHTA) on the reported instances. We
observe that algorithm (RCD-IHT-$u^e$) performs very few full
iterations in order to attain best function value amongst all three
algorithms. Moreover, the number of full iterations performed by
algorithm (RCD-IHT-$u^e$) scales up very well with the dimension of
the problem.

\renewcommand{\tabcolsep}{4pt}
\begin{table}[ht]
\centering \caption{Performance of Algorithms (IHTA),
(RCD-IHT-$u^q$), (RCD-IHT-$u^e$)} {\small \label{tabel1}
\begin{tabular}{|c|c|c|c|c|c|c|c|c|c|}
\hline
$m\backslash n$  &\multicolumn{3}{c|}{\textbf{(IHTA)}} & \multicolumn{3}{c|}{\textbf{(RCD-IHT-$u^q$)}} &
\multicolumn{3}{c|}{\textbf{(RCD-IHT-$u^e$)}}\\
\cline{2-10}
 & $F^*$ & $\norm{x^*}_0$ & iter & $F^*$ & $\norm{x^*}_0$ & full-iter & $F^*$ & $\norm{x^*}_0$ & full-iter\\
\hline
\hline
$20\backslash 100$ & 1.56 & 23 & 797 & 1.39 & 21 & 602 & -0.67 & 15 & 12 \\
\hline
$50\backslash 100$ & -95.88 & 31 & 4847 & -95.85 & 31 & 4046 & -449.99 & 89 & 12 \\
\hline
$30\backslash 200$ & -14.11 & 35 & 2349 & -14.30 & 33 & 1429 & -92.95 & 139 & 12  \\
\hline
$50\backslash 200$ & -0.88 & 26 & 3115 & -0.98 & 25 & 2494 & -13.28 & 83 & 19  \\
\hline
$70\backslash 300$ & -12.07 & 70 & 5849 & -11.94 & 71 & 5296 & -80.90 & 186 & 19 \\
%\hline
%$70\backslash 500$ &  -4.07 & 95 & 5744 & -3.90 & 96 & 5387 & -29.57 & 200 & 14 \\
\hline
$70\backslash 500$ &  -20.60 & 157 & 6017 & -19.95 & 163 & 5642 & -69.10 & 250 & 16 \\
\hline
$100\backslash 500$ & -0.55  & 16 & 4898 & -0.52 & 16 & 5869 & -47.12 & 233 & 14 \\
\hline
$80\backslash 1000$ & 13.01 & 197 & 9516 & 13.71 & 229 & 7073 & -0.56 & 19 & 13 \\
\hline
$80\backslash 1500$ & 5.86 & 75 & 7825 & 6.06 & 77 & 7372 & -0.22 & 24 & 14 \\
%\hline
%$100\backslash 1500$ & 22.23 & 340 & 11815 & 27.17 & 391 & 10870 & -0.75 & 34 & 14 \\
%\hline
%$150\backslash 2000$ & 26.43 & 377 & 21639 & 36.30 & 464 & 19600 & -0.40 & 25 & 13 \\
\hline
$150\backslash 2000$ & 26.43 & 418 & 21353 & 25.71 & 509 & 20093 & -30.59 & 398 & 16 \\
\hline
$150\backslash 2500$ & 26.52 & 672 & 15000 & 27.09 & 767 & 15000 & -55.26 & 603 & 17 \\
\hline
\end{tabular}
}
\end{table}
\vspace{5pt}

%%%%%%%%%%%%%%%%%%%%%%%%%%%%%%%%%%%%%%%%%%%%%%%%%%%%%%%%%%%%5
%%%%%%%%%%%%%%%%%%%%%%%%%%%%%%%%%%%%%%%%%%%%%%%%%%%%%%%%%%%%%%%%%
%
%
%
%\section{Conclusions}
%In this paper we have derived efficient algorithms for solving  $\ell_0$ regularized
%optimization problems, where the objective function is composed of a
%smooth convex function and the $\ell_0$ regularization. We have analyzed
%necessary optimality conditions for this nonconvex problem which
%lead to the separation of the local minima  into two restricted
%classes. Based on these restricted classes of local minima we have devised
%random coordinate descent type methods  for solving our problem. We have proved that
% for each algorithm   any limit point  is a local minimum from one of these
%restricted classes of local minimizers.  Under the strong
%convexity assumption we have also proved linear convergence in probability for
%both methods. Finally, we have performed  several numerical experiments which showed  the
%superior behaviour  of our methods in comparison with the usual
%iterative hard thresholding algorithm.

%%%%%%%%%%%%%%%%%%%%%%%%%%%%%%%%%%%%%%%%%%%%%%%%%5
%%%%%%%%%%%%%%%%%%%%%%%%%%%%%%%%%%%%%%%%%%%%%%%%%%%%%%

\end{sloppy}
\end{document}